\documentclass{amsart}

\usepackage{amsmath,amssymb}
\usepackage{xcolor}
\usepackage{graphicx}
\usepackage{rotating,subfigure}
\usepackage{multirow}
\usepackage{mathtools}
\usepackage{algorithm,algorithmic}
\usepackage{mathrsfs}
\usepackage{dsfont}
\usepackage{color}
\usepackage{comment,url}

\newtheorem{thm}{Theorem}[section]

\newtheorem{lem}[thm]{Lemma}
\newtheorem*{lem*}{Lemma}
\newtheorem{prop}[thm]{Proposition}
\theoremstyle{definition}

\theoremstyle{remark}
\newtheorem*{rem}{Remark}
\newtheorem*{thm*}{Theorem}

\begin{document}

\title[PRSA analysis]{Theoretical analysis of phase-rectified signal averaging (PRSA) algorithm}

\author{Jiro Akahori}
\address{Department of Mathematical Sciences, Ritsumeikan University, Shiga, 525-8577, Japan}
\email{jiro.akahori@gmail.com}

\author{Joseph Najnudel}
\address{School of Mathematics, University of Bristol, BS8 1TW, Bristol, United Kingdom}
\email{joseph.najnudel@bristol.ac.uk}

\author{Hau-Tieng Wu}
\address{Courant Institute of Mathematical Sciences, New York University, New York, NY, 10012, USA}
\email{hauwu@cims.nyu.edu}

\author{Ju-Yi Yen}
\address{Department of Mathematical Sciences, University of Cincinnati, Cincinnati, OH 45221, USA}
\email{ju-yi.yen@uc.edu}
\thanks{J.-Y. Yen is  is grateful to the Mathematics Division of the National Center for Theoretical Sciences (Taiwan) for their hospitality and support during some extended visits.}

\begin{abstract}
Phase-rectified signal averaging (PRSA) is a widely used algorithm to analyze nonstationary biomedical time series. The method operates by identifying hinge points in the time series according to prescribed rules, extracting segments centered at these points (with overlap permitted), and then averaging the segments. The resulting output is intended to capture the underlying quasi-oscillatory pattern of the signal, which can subsequently serve as input for further scientific analysis.
However, a theoretical analysis of PRSA is lacking. In this paper, we investigate PRSA under two settings.
First, when the input consists of a superposition of two oscillatory components, $\cos(2\pi t)+A\cos(2\pi (\xi t+\phi))$, where $A>0$, $\xi\in (0,1)$ and $\phi\in [0,1)$, we show that, asymptotically when the sample size $n\to \infty$, the PRSA output takes the form $A'\sin(2\pi t)+B'\sin(2\pi \xi t)$, where $A',B'\neq 0$. 
Second, when the input is a stationary Gaussian random process, we establish a central limit theorem: under mild regularity conditions, the averaged vector produced by PRSA converges in distribution to a Gaussian random vector as $n\to \infty$ with mean determined by the covariance structure of the random process. These results indicate that caution is warranted when interpreting PRSA outputs for scientific applications.
\end{abstract}

\maketitle

\section{Introduction}
Nonstationary time series from complex systems over extended periods, such as heartbeats, respiration or brain waves, usually experience continuous internal and external influences that disrupt their periodic behavior and reset regulatory processes, causing oscillatory desynchronization. This results in signals exhibiting putative ``quasi-periodic'' characteristics with multiple periodic segments that is often contaminated by undesired noises. In addition to the mathematical definition of quasi-periodicity that a quasi-periodic signal is the superposition of several periodic functions with incommensurate (irrationally related) frequencies \cite{levitan1982almost,arnol2013mathematical,amerio2013almost}, in practice quasi-periodicity is also referred to qualitatively as ``approximately but not exactly repeating'' behavior. In this case, conventional approaches such as spectral analysis face inherent limitations in characterizing such signals, largely due to their quasi-periodic nature and the absence of explicit phase-reset mechanisms.

Phase-rectified signal averaging (PRSA) \cite{bauer2006phase} is an algorithm specifically designed to analyze such challenging time series, with its core objective being the quantification of quasi-periodic structures masked by the non-stationary nature of composite signals and noise. 
The PRSA algorithm is composed of three steps. 
The first step is identifying the main repetitive pattern (or called quasi-periodicity in \cite{bauer2006phase}) within the given time series $x_i$, $i=1,\ldots,n$, even if the signal is noisy or irregular, by finding {\em hinge point} (or called {\em coherence time} in \cite{bauer2006phase}). Denote the hinge points as $\mathcal{H}\subset\{1,\ldots,n\}$. Once identified, a local segment of the signal centered at each hinge point is cut with a predefined length $L\in \mathbb{N}$, which is called a ``cycle''; that is, for $j\in \mathcal{H}$, $X_j:=\{x_{j-L},\ldots,x_{j+L}\}\in \mathbb{R}^{2L+1}$. The second step is collecting these cycles and aligning them according to the hinge points. These synchronized cycles are averaged by $z_{n,L}:=\frac{1}{|\mathcal{H}|}\sum_{j\in \mathcal{H}}X_{j}$ to reduce random noise present in each individual cycle. The resulting $z_{n,L}$ is expected to be the hidden pattern that repeats itself inside the time series. The third step is quantifying the behavior of $z_{n,L}$ by applying any proper nonstationary time series analysis tools. 
The authors claim that PRSA allows users to {\em quantify the typical coherence time for each quasi-periodicity and to separate processes occurring during increasing and decreasing parts of the signal}, which is implemented in identifying the underlying ``periodic segments'' of the time series in Step 1, and averaging these segments to reveal coherent patterns within the quasi-periodic signal structure in Step 2. 
It is widely claimed and believed that PRSA captures characteristic quasi-periodicities, short-term correlations, and time inversion asymmetry (causality), while mitigating non-stationarities and noise.

Since its introduction, PRSA has been applied to study various biomedical time series. For example, acidaemia at birth from fetal heart rate \cite{georgieva2014phase}, heart rate during atrial fibrillation attack \cite{lemay2008phase}, EEG as an anesthesia depth monitor \cite{liu2017quasi}, autonomic changes with aging \cite{campana2010phase}, mortality after acute myocardial infarction \cite{kisohara2013multi}, quantification of cardiac vagal modulation \cite{nasario2014refining}, to name but a few. 
However, to our knowledge, a theoretical framework that rigorously characterizes the behavior of PRSA is still lacking, aside from some ad hoc arguments \cite{bauer2006phase}. For example, although we agree with the qualitative statement that the hinge points, or coherence time, may be related to each quasi-periodicity, the precise conditions under which this relationship holds remain unclear. It is stated in \cite[p.~427]{bauer2006phase} that {\em ..., the anchor points are determined from the signal itself. Hence, the anchor points are always phase synchronized with the signal, even if the phase is unstable or non-stationarities occur.} While this intuition is appealing, a rigorous argument is needed to fully understand how PRSA operates.
Since PRSA has been widely applied in scientific studies with promising empirical results, it is therefore important to establish a solid theoretical foundation under well-defined mathematical models, which is the main objective of this paper. We also note that PRSA has been extended to multivariate time series \cite{schumann2008bivariate,bauer2009bprsa,muller2012bivariate}; however, in this work, we restrict our analysis to the univariate setting.

In this paper, we analyze PRSA under two mathematical models, motivated by questions in statistical inference and its applications to biomedical signal analysis. 
First, we study a simple deterministic signal with two oscillatory components, $\cos(2\pi t)+A\cos(2\pi (\xi t+\phi))$, where $A>0$, $\xi\in (0,1)$ and $\phi\in [0,1)$. We show that asymptotically when the sample size $n\to \infty$, $z_{n,L}$ is a discretization of $A'\sin(2\pi t)+B'\sin(2\pi \xi t)$, where $A',B'\neq 0$. See Theorem \ref{theorem deteministic discrete}. It is worth noting that, although a precise mathematical definition is not provided in \cite{bauer2006phase}, the notion of quasi-periodicity considered therein appears to be more general than the classical mathematical formulations \cite{levitan1982almost,arnol2013mathematical,amerio2013almost}. In particular, the oscillatory components in \cite{bauer2006phase} may exhibit time-varying amplitudes and frequencies, making them substantially more complex than the standard quasi-periodic functions studied in analysis. The first model we consider represents a simplified version of this scenario, yet, as we shall demonstrate below, it remains far from trivial.

Second, we consider a null setting in which the input time series $x_i$ is a stationary random process. Under mild regularity conditions that the covariance function decays to 0, we establish a law of large number in Proposition \ref{lawlargenumbers} that $z_{n,L}$ converges in probability to a nontrivial vector $\left(\frac{\mathbb{E}[ x_{\ell} \mathds{1}_{w_0 > c} ]}{\mathbb{E}[ \mathds{1}_{w_0 > c} ] }\right)_{\ell=-L}^L$ with a precise expression in \eqref{expression of zeta in LNN}. If we further assume that the covariance function decays to 0 sufficiently fast, we obtain a central limit theorem in Theorem \ref{central limit theorem for prsa} that $\sqrt{n} z_{n,L}$  converges in distribution to a Gaussian vector as $n\to \infty$. The main technical challenge in proving the theorem lies in understanding how the criterion $w_i>c$, which effectively acts as a weighting mechanism in \eqref{equation zn formula}, influences the averaging toward $z_{n,L}$. In the deterministic setting, this difficulty is addressed using Weyl's equilibrium lemma, while in the stochastic setting we carefully leverage the Gaussian properties and establish a cumulant control. 
These results highlight fundamental features of PRSA. Even in the simple setting of a signal composed of fixed-frequency oscillations with constant amplitudes and no noise contamination, the algorithm might provide disturbed quasi-periodic output. On the other hand, in the pure-noise case, the algorithm can still produce nontrivial quasi-periodic like outputs depending on the covariance structure of the noise. Both phenomena may lead to misleading interpretations. Thus, although PRSA has been shown to yield useful insights in biomedical signal analysis, caution is warranted when interpreting its outputs, particularly when treating them as quasi-   periodic signals for scientific investigation.

The remainder of this paper is organized as follows.
Section~\ref{section PRSA algorithm} provides a concise summary of the PRSA algorithm.
In Section~\ref{section PRSA det}, we analyze the algorithm under a deterministic two-harmonic oscillatory model.
Section~\ref{section PRSA sto} extends the analysis to stationary Gaussian random processes.
In Sections \ref{section PRSA det} and \ref{section PRSA sto}, numerical simulations are presented to validate the theoretical results and to illustrate the intricate behavior of PRSA in various settings. The Matlab code to reproduce these results can be found in \url{https://github.com/hautiengwu2/PRSA.git}.

\section{PRSA algorithm and mathematical model}\label{section PRSA algorithm}

\subsection{PRSA algorithm}
We describe the PRSA algorithm in the general framework.
Consider a time series $x_i\in \mathbb{R}$, $i\in\{1,2,\ldots,n\}$, and a set of {\em decision making functions} modeled by measurable functions $f_i:\mathbb{R}^n\to \mathbb{R}$. Fix $L\in \mathbb{N}$ to be the window length associated with the hidden pattern of interest.

\begin{enumerate}
\item Find $i\in \{1,\ldots,n\}$ so that {\em predefined} conditions are satisfies; that is, $f_i(x_1,\ldots,x_{i-1},x_{i},x_{i+1},\ldots,x_n)$ satisfies some conditions. Collect all these indices as $\mathcal{H}_{c,n}\subset \{1,\ldots,n\}$, which is called the {\em hinge point} set; that is, 
\[
\mathcal{H}_{c,n}:=\{i|\, f_i(x_1,\ldots,x_{i-1},x_{i},x_{i+1},\ldots,x_n) \mbox{ satisfies some conditions}\}\,.
\] 

\item For each $i\in \mathcal{H}_{c,n}\cap\{L+1,\ldots,n-L\}$, collect segments 
\[
X_{i,L}:=[x_{i-L},x_{i-L+1},\ldots, x_{i+L-1},x_{i+L}]^\top\in \mathbb{R}^{2L+1}\,,
\]
and evaluate 
\[
z_{n,L}=\frac{1}{|\mathcal{H}_{c,n}|}\sum_{i\in \mathcal{H}_{c,n}} X_{i,L}\in \mathbb{R}^{2L+1}\,.
\]

\item Quantify the behavior of $z_{n,L}$ by a chosen signal processing tool. 
\end{enumerate}

In practice, we can design general hinge points and complicated rules. For example, if $x_i$ is a electrocardiogram, $f_i$ could be the algorithm detecting the R peaks, and we can impose that $f_i$ satisfies some conditions. 
As the first exploration, in this paper, we focus our exploration on the simplest one that is commonly used in practice \cite{bauer2006phase,kantelhardt2007prsa}; that is, we choose the decision making functions $f_i$ to be 
\begin{align}\label{definition rule of fi}
f_i(\ldots,x_{i-1},x_{i},x_{i+1},\ldots):=x_i-x_{i-1}
\end{align}
and the rules to be 
\begin{align}
f_i(\ldots,x_{i-1},x_{i},x_{i+1},\ldots)>c
\end{align}
for some predetermined $c\in \mathbb{R}$. In this case, $\mathcal{H}_{n,c}$ is determined from the first order finite difference of $x_i$, denoted as $w_i:=x_i-x_{i-1}$, and $\mathcal{H}_c$ is the level set of $w_i$ above $c$. Usually, $c$ is chosen to be $0$, but a different $c$ could also be considered. $L\in \mathbb{N}$ is  a sufficiently large integer (like $10$). Our goal is exploring the asymptotic behavior of $z_{n,L}$ when $N\to \infty$, under proper models. 

\begin{rem} We shall remark that in practice, researchers quantify the dynamics of $z_{n,L}$ using various signal processing tools to generate a summary index for the nonstationary time series under exploration. For example, a commonly applied index 
\[
c_{n,L}:=\frac{1}{L-\lfloor L/2\rfloor+1}\sum_{j=\lfloor L/2\rfloor}^L z_{n,L}(j)- \frac{1}{L-\lfloor L/2\rfloor+1}\sum_{j=-L}^{-\lfloor L/2\rfloor}z_{n,L}(j)\,.
\]
is a coefficient of the wavelet transform of $z_{n}$ with a simple ``Haar-like'' wavelet. 
$c_{n,L}$ is used as a feature to study the spectral content associated with the nonstationarity of a given time series. With the established statistical behavior of $z_{n,L}$, the statistical behavior of $c_{n,L}$ can be explored. 
\end{rem}

\begin{rem}
We note that for time series sampled at high frequency, it is reasonable to approximate the data using a continuous model. The PRSA algorithm can be naturally extended to this setting.
Consider a continuous function $g\in C([0,T])$, where $T>0$, and a set of {\em decision making functions} modeled by measurable functions $f_y:C([0,y])\to \mathbb{R}$. Fix $L>0$ to be the window length associated with the hidden pattern of interest.
%\begin{enumerate}
%\item 
First, find $y\in [0,T]$ so that {\em predefined} conditions are satisfies; that is, $f_y(g|_{[0,y]})$ satisfies some conditions. Collect all these indices as $\mathcal{H}_{T}\subset [0,T]$, which is called the {\em hinge point} set
%\item 
Second, suppose $\mathcal{H}_T$ is measurable. Construct a function $z_{L}:[-L,L]\to \mathbb{R}$ by evaluating 
%\[
$z_{L}(z)=\frac{1}{|\mathcal{H}_T|}\int_{\mathcal{H}_T} f(t+z)dt$.
%\] 
%\item 
Third, quantify the behavior of $z_{L}$ by a chosen signal processing tool. 
%\end{enumerate}
%
In parallel to the PRSA algorithm we analyze in this work, the function $g\in C^1([0,T])$, the decision making functions $f_y$ is 
%\begin{align}\label{definition rule of fy}
$f_y(g|_{[0,y)}):=g'(y)$,
%\end{align}
and the rule is
%\begin{align}
$f_y(g|_{[0,y)})>c$
%\end{align}
for some predetermined $c\in \mathbb{R}$. %In this case, $\mathcal{H}_{T}$ is measurable. %Usually, $c$ is chosen to be $0$, but a different $c$ could also be considered. Our goal is 
We can thus explore the asymptotic behavior of $z_{L}$ when $T\to \infty$ under proper assumptions. 
\end{rem}

\subsection{Relationship with other algorithms}

PRSA is closely related to several existing algorithms in time series analysis. A fundamental idea underlying PRSA is the existence of a repetitive pattern in the signal and the assumption that such a pattern can be effectively extracted through the use of hinge points and averaging. {The concept of quantifying the recurrence has also been widely explored in biomedical signal processing. For example, recurrence quantification analysis (RQA) \cite{zbilut1992embeddings} quantifies the number and duration of recurrences in a time series to characterize the underlying system's dynamics. Typically, RQA is performed by reconstructing the system's trajectory in phase space using Taken's embedding theorem \cite{takens2006detecting}. The approach by which RQA identifies similar patterns differs from that of PRSA, which relies on specific anchor conditions (such as ascending or descending points), suggesting potential complementarity between the two methods.

The idea of identifying and recovering hidden repetitive patterns is also closely related to signal decomposition algorithms based on wave-shape functions (WSFs) \cite{HTWu2013}, which are 1-periodic functions representing the repetitive waveform structure. WSFs play a central role in modeling oscillatory signals composed of multiple non-sinusoidal components. By applying time-frequency analysis techniques to extract the instantaneous frequency \cite{CYLin2018} and phase, hinge points can be determined from the phase information, thereby guiding the design of signal decomposition schemes.
For instance, if the hinge points correspond to the R-peaks in a maternal ECG signal and the window length $L$ is chosen sufficiently large to cover a complete cardiac cycle, the resulting $z_{n,L}$ provides a representative template of maternal cardiac cycles. In more general cases where the WSF is not unique, meaning that multiple repetitive patterns coexist, a manifold structure can be introduced to parametrize all possible WSFs, referred to as the WSF manifold \cite{WSFmanifold2021}. This manifold model motivates the use of low-rank structures for signal decomposition and denoising. In particular, by combining matrix denoising with optimal shrinkage, this framework has been successfully applied to separate fetal ECG from trans-abdominal maternal ECG recordings \cite{su2025data} and other applications.}

The template construction step via averaging, $z_{n,L}:=\frac{1}{|\mathcal{H}|}\sum_{j\in \mathcal{H}}X_{j}$, is inherently linked to kernel regression \cite{hardle1992kernel}. Specifically, if the hinge points are defined by the criterion $w_i:=x_{i}-x_{i-1}>c$ for some $c\in \mathbb{R}$, we can rewrite 
\begin{align}\label{equation zn formula}
z_{n,L}=\frac{\sum_{j}X_{j}\mathds{1}_{w_j>c}}{\sum_j \mathds{1}_{w_j>c}}\in \mathbb{R}^{2L+1}.
\end{align}
By interpreting $\mathds{1}_{w_j>c}$ as a data-driven kernel that depends on the level set of the finite difference of $x_i$, PRSA can be directly associated with Nadaraya-Watson kernel regression \cite{nadaraya1964estimating,watson1964smooth}, but with a highly nontrivial kernel. 
{Recall that in a standard regression problem, one considers a dataset $\{x_i,y_i\}_{i=1}^n\subset \mathbb{R}^d\times \mathbb{R}^{d'}$, where the {\em response} $y_i$ and {\em predictor} $x_i$ are linked by a {\em regression function $f\in C^1(\mathbb{R}^d,\mathbb{R}^{d'})$} via $y_i=f(x_i)+\epsilon_i$ with centered noise $\epsilon_i$ of finite covariance. Typically, $x_i$ are sampled independently following some distribution, and are independent of $\epsilon_i$. In practice, the goal is to estimate $f(x)$ for a given $x\in \mathbb{R}^d$. The Nadaraya-Watson estimator is defined as $\hat{f}(x):=\frac{\sum_{j=1}^nK\big(\frac{\|x-x_j\|}{h}\big)y_j}{\sum_{j=1}^nK\big(\frac{\|x-x_j\|}{h}\big)}$, where $K$ is a chosen kernel and $h>0$ is called the bandwidth. To see how \eqref{equation zn formula} relates to the Nadaraya-Watson estimator, suppose that the hidden pattern we aim to recover with PRSA can be modeled as $f(x_i)\in \mathbb{R}^{2L+1}$, where $x_i$ is the predictor encoded in the observed time series that we cannot directly access, and the noisy response is $X_i=f(x_i)+\epsilon_i$. Under this heuristic model, $\mathds{1}_{w_j>c}$ can be viewed as a highly nonlinear and predictor-dependent kernel without bandwidth within the Nadaraya-Waston framework. It should be emphasized, however, that because of the strong dependence between the predictor and noise, the nonlinear nature of the regression function, and the complexity of the implicit kernel, this analogy is only heuristic and not meant as a rigorous correspondence.}

Furthermore, the criterion $w_i:=x_{i}-x_{i-1}>c$ is reminiscent of the approach used in topological data analysis (TDA) to study time series \cite{chung2021persistent}. To construct the persistent homology of a time series to quantify its topological structure, a common idea is to build a filtration by varying the threshold $c$; that is, $F_c=\{j|w_j>c\}$ so that $F_c\subset F_{c'}$ when $c'<c$. The evolution of $F_c$ as $c$ varies plays a critical role in characterizing this filtration, and hence the associated persistent homology. Since the set $F_c$ is directly involved in \eqref{equation zn formula}, the analysis of PRSA is directly related to understanding how TDA can be applied to time series. 
These connections of PRSA with other algorithms reflect the rich structure underlying PRSA.

\subsection{Mathematical model}
In this paper, we adapt the widely used mathematical definition of quasi-periodicity \cite{levitan1982almost,arnol2013mathematical,amerio2013almost}. With this perspective, we introduce two mathematical models to investigate the behavior of the PRSA algorithm. 

The first model examines PRSA with the decision-making function \eqref{definition rule of fi} in a deterministic setup, where the input signal consists of two discretized sinusoidal functions; that is, $x_n$ is a discretization of $f(t)=\cos(2\pi t)+A\cos(2\pi(\xi t+\phi))$, where $A>0$, $\xi\in (0,1)$ and $\phi\in [0,1)$. When $\xi$ is irrational, $f$ is quasi-periodic by definition \cite{levitan1982almost,arnol2013mathematical,amerio2013almost}. Here, we analyze the asymptotic behavior of $z_{n,L}$ as $n \to \infty$ with $L$ fixed. For the deterministic setup, we also derive the continuous version for a comparison. 

The second model investigates the asymptotic behavior of PRSA with the decision-making function \eqref{definition rule of fi} under a stationary stochastic framework; that is,
consider a centered stationary Gaussian process \( (x_n)_{n \in \mathbb{Z}} \) with covariance function \( C(k) = \mathbb{E}[x_0 x_k] \) such that $C(0) > 0$. In this setting, we derive a central limit theorem for the random vector $z_{n,L}$ as $n \to \infty$ with $L$ fixed.

Together, these two models provide insight into the complex and nonlinear behavior of PRSA. We note that when a signal arises from an explicit mechanical system, such as those based on multi-periodic motions in integrable Hamiltonian systems \cite{arnol2013mathematical}, the discussion could be extended. However, taking such mechanistic formulations into the analysis falls beyond the scope of the present work.

\section{PRSA of deterministic signals} \label{section PRSA det}

To gain intuition for the behavior of the PRSA statistic, it is helpful to interpret the underlying mechanism as a form of wave superposition. 
Consider the two harmonic components model, $f(t)=\cos(2\pi t)+A\cos(2\pi(\xi t+\phi))$, where $A> 0$, $\xi\in (0,1)$ and $\phi\in [0,1)$. When $A\to 0$, this model is reduced to the single harmonic component model. The goal is to show that when there are multiple harmonic components in the signal, we may need to be careful when we interpret the PRSA's original goal of detecting quasi-periodicities in non-stationary data, since the result might be misleading.

The signal averaging performed by PRSA can be seen as aligning and averaging portions of the signal that share similar local dynamics. In this paper, we specifically consider instances where the incremental difference $w_m = x_m - x_{m-1}$ exceeds a threshold. This thresholding acts as a phase selector, and the subsequent averaging resembles the interference pattern of waves with partially synchronized phases. 
Such a perspective suggests that $z_{n,\ell}$ captures coherent structures in the signal by amplifying recurring directional patterns, much like constructive interference amplifies signals in wave theory. This view aligns with the original motivation of PRSA as a tool for extracting quasi-periodic features from noisy time series.
%}
We establish the following theorem.

\begin{thm}\label{theorem deteministic discrete}
Consider the deterministic signal given by 
\[
x_m = \cos(2\pi \xi_1 m + \varphi_1) + A \cos(2\pi \xi_2 m  + \varphi_2), 
\]
where $m \in \mathbb{Z}$, 
with amplitude $A>0$, frequencies $\xi_1,\xi_2 \in (0,1)$ and phases $\varphi_1, \varphi_2 \in \mathbb{R}$. We assume that $1, \xi_1, \xi_2$ 
are linearly independent over $\mathbb{Q}$, and that
$$ c< 2 [\sin(\pi \xi_1) + A \sin (\pi \xi_2)].$$
Then, for fixed $L \geq 1$, $-L \leq \ell \leq L$, as $n\to \infty$, we have \[
z_{n,L}(\ell)\to B_1 \sin ( \pi \xi_1 (2 \ell + 1)) + B_2 \sin ( \pi \xi_2 (2 \ell + 1))
\]
for $\ell=-L,\ldots,L$, where 

$$B_1 = \frac{ \int_{\max\left(-1,\frac{c/2 - \sin(\pi \xi_1)}{A \sin (\pi \xi_2)} \right)}^{\min\left(1,\frac{c/2 + \sin(\pi \xi_1)}{A \sin (\pi \xi_2)} \right)} \, \sqrt{\frac{   1 - \left(\frac{c/2 - A u \sin (\pi \xi_2)}{\sin(\pi \xi_1)} \right)^2 }{1 - u^2}} du }{\int_{-1}^1 \int_{-1}^1 \frac{1}{\sqrt{(1 - u^2)(1-v^2)}} \mathds{1}_{u \sin (\pi \xi_1) +A v \sin (\pi \xi_2) > c/2 } du \, dv}\,, 
$$
$$B_2 = \frac{ \int_{\max\left(-1,\frac{c/2 - A\sin(\pi \xi_2)}{ \sin (\pi \xi_1)} \right)}^{\min\left(1,\frac{c/2 + A\sin(\pi \xi_2)}{ \sin (\pi \xi_1)} \right)} \, \sqrt{\frac{   1 - \left(\frac{c/2 -  u \sin (\pi \xi_1)}{A\sin(\pi \xi_2)} \right)^2 }{1 - u^2}} du }{\int_{-1}^1 \int_{-1}^1 \frac{1}{\sqrt{(1 - u^2)(1-v^2)}} \mathds{1}_{u \sin (\pi \xi_1) +A v \sin (\pi \xi_2) > c/2 } du \, dv},
$$
where the denominator does not vanish when $c$ is properly chosen. %because of the assumption made on $c$. 
In the particular case where $c = 0$, these expressions simplify to 
$$
B_1 = \frac{4}{\pi^2} \int_0^{\min (1, 1/(A \xi))}
\sqrt{ \frac{ 1 - A^2 \xi^2 u^2}{1 - u^2}} du\,,
$$
for $\xi := \frac{\sin (\pi \xi_2)}{\sin (\pi \xi_1)}$, and
$$
B_2= \frac{4 A^2 \xi}{\pi^2} \int_0^{\min (1, 1/(A \xi))}
\sqrt{ \frac{ 1 -  u^2}{1 - A^2 \xi^2  u^2}} du.
$$
\end{thm}

Note that in the case where $c = 0$, the expressions for $B_1$ and $B_2$ correspond to elliptic integrals of the second kind. Clearly, the relatively phase information $\varphi_2-\varphi_1$ disappears, and the quasi-periodic behavior of the signal is nonlinearly perturbed by PRSA. We can see clearly the nontrivial behavior of PRSA even in this simple two harmonic components model without noise contamination, which warrants that we shall be careful when interpret $z_{n,L}$ in practice. 

\begin{proof}
We have
\begin{align*}
w_m = \,& x_{m} - x_{m-1}
\\ = \,& [\cos(2\pi \xi_1 m + \varphi_1)
- \cos(2\pi \xi_1 (m-1)+ \varphi_1)]
\\ & + A [\cos(2\pi \xi_2 m + \varphi_2)
- \cos(2\pi \xi_2 (m-1)+ \varphi_2)]
\\ = \,& - 2 \sin ( \pi \xi_1 (2m-1) + \varphi_1) \sin ( \pi \xi_1 )
- 2 A \sin ( \pi \xi_2 (2m-1) + \varphi_2) \sin ( \pi \xi_2 ). 
\end{align*}
Thus, the hinge condition $w_m>c$ is determined by linear combination of sinusoidal terms 
involving both $\xi_1$ and $\xi_2$, weighted by the factors $\sin(\pi \xi_1)$ and $\sin(\pi \xi_2)$.
This makes explicit the role of the oscillatory structure in selecting hinge points.

For $\ell \in \mathbb{Z}$ and 
$n \geq 1$, consider the ratio 
$$
z_{n}(\ell) := \frac{\sum_{m = -n}^n x_{m+ \ell} \mathds{1}_{w_m > c}} {\sum_{m = -n}^n  \mathds{1}_{w_m > c}}
= \frac{\frac{1}{2n+1}\sum_{m = -n}^n x_{m+ \ell} \mathds{1}_{w_m > c}} {\frac{1}{2n+1}\sum_{m = -n}^n  \mathds{1}_{w_m > c}}\,.
$$
We now introduce random variables as a tool for proving convergence results, but we remind that the setting is deterministic. Denote the uniform probability measure on 
the finite set $\{-n,-n+1, \dots, n-1,n\}$ as $\mathbb{P}_n$, 
$\mathbb{E}_n$ denotes the expectation under $\mathbb{P}_n$, and $M_n$ is a random variable 
on $\{-n,-n+1, \dots, n-1,n\}$ which follows the distribution $\mathbb{P}_n$. With this notation, we have
$$
z_{n}(\ell) = \frac{ \mathbb{E}_n [ x_{M_n+ \ell} \mathds{1}_{w_{M_n} > c}]}{ \mathbb{P}_n (w_{M_n} > c)}\,.
$$
Now, let us define the following random variables:

$$U_n = \pi \xi_1 (2M_n - 1) + \varphi_1, \; V_n =  \pi \xi_2 (2M_n-1) + \varphi_2.$$
We have $w_{M_n} > c$ if and only if 
$$\sin(U_n) \sin(\pi \xi_1) 
+ A \sin(V_n) \sin(\pi \xi_2) < -c/2.$$
Hence, the hinge-point selection is given by a deterministic 
function of the pair $(U_n, V_n)$.
Moreover,
$$x_{M_n+ \ell} 
= \cos( U_n +  \pi \xi_1 (2\ell+1)) + A \cos( V_n +  \pi \xi_2 (2\ell+1))$$
is also a deterministic function of $U_n$ and $V_n$.

\begin{comment}
{ Since $1,\xi_{1},\xi_{2}$ are linearly independent over $\mathbb{Q}$, by the
multidimensional Weyl–Kronecker equidistribution theorem, the sequence
\[
(\{M_n\xi_1\},\{M_n\xi_2\})
\]
is equidistributed on $[0,1]^2$,
equivalently the phases
\[
U_n :=2\pi \{M_n\xi_1\}, \qquad V_n :=2\pi \{M_n\xi_2\}
\]
converge in distribution to independent $\mathrm{Unif}[0,2\pi)$ variables
$U,V$.  (The term ``equidistributed” means that the empirical measure of
$(U_n,V_n)$ converges to the constant-density Lebesgue measure on the
two-torus; some authors call this ``uniform distribution mod~1”.)
Consequently,
\begin{align*}
&\frac1{2n+1}\sum_{m=-n}^n
x_{m+\ell}\mathbf 1_{\{w_m>c\}}
= \mathbb{E}_U\Bigl[x_{M_n+\ell}\mathbf 1_{\{w_{M_n}>c\}}\Bigr]\longrightarrow\\
&
\mathbb{E}_{U,V}\Bigl[\bigl(\cos(U+\pi\xi_1(2\ell+1))
  +A\cos(V+\pi\xi_2(2\ell+1))\bigr)
  \mathbf 1_{\{\sin(U)\sin(\pi\xi_1)
    +A\sin(V)\sin(\pi\xi_2)<-c/2\}}\Bigr],
\end{align*}
where $\mathbb E_{U,V}$ denotes expectation with respect to the product
measure of $U,V\sim\mathrm{Unif}[0,2\pi)$.  
Hence $y_{n,\ell}$ converges, as $n\to\infty$, to the ratio of the above
expectations under the constant-density (Lebesgue) measure on the torus.}
\end{comment}
%%%%%%%%%%%%%%%%%%%%%%%%%%%%%%%%%%%

Since we assume that $1,\xi_1,\xi_2$ are linearly independent over $\mathbb{Q}$, by the multidimensional Weyl–Kronecker equidistribution theorem \cite{montgomery1994ten}, the sequence of random variables 
\[
(U_n\!\!\!\!\mod 2\pi, V_n\!\!\!\!\mod 2\pi)_{n \geq 1}
\]
converges in distribution to a random variable $(U,V)$ on $[0,2\pi)^2$, with probability distribution equal to $1/4 \pi^2$ times the Lebesgue measure on $[0, 2 \pi)^2$. 
We denote $\mathbb{P}_{U,V}$ this probability measure and $\mathbb{E}_{U,V}$ the expectation under $\mathbb{P}_{U,V}$.

We deduce the convergence: 
\begin{align*} 
& \mathbb{E}_n[ x_{M_n + \ell} 
\mathds{1}_{w_{M_n} > c}] \\ 
 \underset{n \rightarrow \infty}{\longrightarrow} \,&
\mathbb{E}_{U,V} [ (\cos( U +  \pi \xi_1 (2\ell+1)) + A \cos( V +  \pi \xi_2 (2\ell+1)))
\mathds{1}_{\sin(U) \sin(\pi \xi_1) 
+ A \sin(V) \sin(\pi \xi_2) < -c/2}]
\end{align*}
and similarly
$$
\mathbb{P}_n(w_{M_n} > c) 
\underset{n \rightarrow \infty}{\longrightarrow}
\mathbb{P}_{U,V}( \sin(U) \sin(\pi \xi_1) 
+ A \sin(V) \sin(\pi \xi_2) < -c/2)\,.
$$
Then, $z_{n} (\ell)$ converges, when $n \rightarrow \infty$, to a completely explicit quantity, given by the following conditional expectation: 
$$ 
\mathbb{E}_{U,V} [ \cos( U +  \pi \xi_1 (2\ell+1)) + A \cos( V +  \pi \xi_2 (2\ell+1))
\; \; \big|\sin(U) \sin(\pi \xi_1) 
+ A \sin(V) \sin(\pi \xi_2) < -c/2],  
$$
i.e.

\begin{align*}
&\lim_{n\to \infty}z_{n}(\ell)\\
=\,& \frac{ \mathbb{E}_{U,V} [(\cos( U +  \pi \xi_1 (2\ell+1)) + A \cos( V +  \pi \xi_2 (2\ell+1))) \mathds{1}_{\sin(U) \sin(\pi \xi_1) 
+ A \sin(V) \sin(\pi \xi_2) < -c/2}]}{\mathbb{P}_{U,V} ( \sin(U) \sin(\pi \xi_1) 
+ A \sin(V) \sin(\pi \xi_2) <-c/2)}  \,.
\end{align*}
In the numerator, we expand the cosines and notice that 
$$
\mathbb{E}_{U,V}[ (\cos U) \mathds{1}_{\sin U \sin (\pi \xi_1) + A  \sin V \sin (\pi \xi_2) < -c/2}] = \mathbb{E}_{U,V} [ (\cos V) \mathds{1}_{\sin U \sin (\pi \xi_1) + A  \sin V \sin (\pi \xi_2) < -c/2}] =0  
$$
by the symmetry $U \mapsto \pi - U$, $V \mapsto \pi - V$, and then
$$
\lim_{n\to \infty}z_{n}(\ell)=B_1 \sin ( \pi \xi_1 (2 \ell + 1)) + B_2 \sin ( \pi \xi_2 (2 \ell + 1)) 
$$
for
\begin{align*}
B_1 =\,& - \frac{ \mathbb{E}_{U,V}[ (\sin U) \mathds{1}_{\sin U \sin( \pi \xi_1)  + A \sin V \sin( \pi \xi_2)< -c/2}]}{\mathbb{P}_{U,V} ( \sin(U) \sin(\pi \xi_1) 
+ A \sin(V) \sin(\pi \xi_2) <-c/2)}  \\
=\,&  \frac{ \mathbb{E}_{U,V}[ (\sin U) \mathds{1}_{\sin U \sin( \pi \xi_1)  + A \sin V \sin( \pi \xi_2)> c/2}]}{\mathbb{P}_{U,V} ( \sin(U) \sin(\pi \xi_1) 
+ A \sin(V) \sin(\pi \xi_2) > c/2)}, 
\end{align*}
where in the second equality, we use the symmetric 
$(U,V) \mapsto (-U, -V)$, 
and similarly
$$ B_2 = \frac{ A \mathbb{E}_{U,V}[ (\sin V) \mathds{1}_{\sin U \sin( \pi \xi_1)  + A \sin V \sin( \pi \xi_2)> c/2}]}{\mathbb{P}_{U,V} ( \sin(U)\sin(\pi \xi_1) 
+ A \sin(V) \sin(\pi \xi_2) > c/2)}.$$

Now, for a quantity $a$ which is a Lebesgue-measurable function of $V$, we 
can compute the following conditional expectation 
$$
\mathbb{E}_{U,V} [ (\sin U) \mathds{1}_{\sin U > a} | V] 
= \frac{1}{\pi } \int_{\arcsin a}^{\pi/2} \sin u \, du
= \frac{1}{\pi } [- \cos u]_{\arcsin a}^{\pi/2} 
= \frac{\sqrt{1 - a^2}}{ \pi} 
$$
when $a \in [-1,1]$, 
$$\mathbb{E}_{U,V} [ (\sin U) \mathds{1}_{\sin U > a} | V]  = \mathbb{E}_{U,V} [ \sin U | V]  
= 0$$
when $a < -1$ and 
$$\mathbb{E}_{U,V} [ (\sin U) \mathds{1}_{\sin U > a} | V]  
= 0$$
when $a > 1$. 
In other words, for any $a \in \mathbb{R}$ which is a measurable function of $V$, 
\begin{align}
\mathbb{E}_{U,V} [ (\sin U) \mathds{1}_{\sin U > a} | V] = 
\frac{\sqrt{\max(0,1 - a^2)}}{ \pi}. 
\end{align}
Applying this equality to 
$$a = \frac{c/2 - A \sin V \sin (\pi \xi_2)}{\sin(\pi \xi_1)},$$
where we notice that $\sin (\pi \xi_1) > 0$ since $\xi \in (0,1)$, and using the tower property of expectation in order to discard the conditioning in $V$, we get
\begin{align*} & \mathbb{E}_{U,V}[ (\sin U) \mathds{1}_{\sin U \sin( \pi \xi_1)  + A \sin V \sin( \pi \xi_2)> c/2} ] \\ & = \frac{1}{ \pi} \mathbb{E}_V \left[ \sqrt{\max \left(0, 1 - \left(\frac{c/2 - A \sin V \sin (\pi \xi_2)}{\sin(\pi \xi_1)} \right)^2 \right)}  \right]
\\ & 
= \frac{1}{\pi^2}
\int_{-1}^1  \sqrt{\frac{\max \left(0, 1 - \left(\frac{c/2 - A u \sin (\pi \xi_2)}{\sin(\pi \xi_1)} \right)^2 \right)}{1 - u^2}} du
\\ & = \frac{1}{\pi^2} 
\int_{\max\left(-1,\frac{c/2 - \sin(\pi \xi_1)}{A \sin (\pi \xi_2)} \right)}^{\min\left(1,\frac{c/2 + \sin(\pi \xi_1)}{A \sin (\pi \xi_2)} \right)} \, \sqrt{\frac{   1 - \left(\frac{c/2 - A u \sin (\pi \xi_2)}{\sin(\pi \xi_1)} \right)^2 }{1 - u^2}} du. 
\end{align*}
On the other hand, 

\begin{align*} & \mathbb{P}_{U,V} ( \sin(U) \sin(\pi \xi_1) 
+ A \sin(V) \sin(\pi \xi_2) > c/2) 
\\ & = \frac{1}{\pi^2}\int_{-1}^1 \int_{-1}^1 \frac{du \, dv}{\sqrt{(1 - u^2)(1-v^2)}} \mathds{1}_{u \sin (\pi \xi_1) +A v \sin (\pi \xi_2) > c/2 }. 
\end{align*}
Dividing the two quantities gives, 
when the denominator does not vanish
$$B_1 = \frac{ \int_{\max\left(-1,\frac{c/2 - \sin(\pi \xi_1)}{A \sin (\pi \xi_2)} \right)}^{\min\left(1,\frac{c/2 + \sin(\pi \xi_1)}{A \sin (\pi \xi_2)} \right)} \, \sqrt{\frac{   1 - \left(\frac{c/2 - A u \sin (\pi \xi_2)}{\sin(\pi \xi_1)} \right)^2 }{1 - u^2}} du }{\int_{-1}^1 \int_{-1}^1 \frac{du \, dv}{\sqrt{(1 - u^2)(1-v^2)}} \mathds{1}_{u \sin (\pi \xi_1) +A v \sin (\pi \xi_2) > c/2 } }. 
$$

The constant $B_2$ can be obtained as $B_1$ by exchanging  $U$ and $V$ and the factors $\sin (\pi \xi_1)$ and $A \sin (\pi \xi_2)$, which gives 
$$B_2 = \frac{ \int_{\max\left(-1,\frac{c/2 - A\sin(\pi \xi_2)}{ \sin (\pi \xi_1)} \right)}^{\min\left(1,\frac{c/2 + A\sin(\pi \xi_2)}{ \sin (\pi \xi_1)} \right)} \, \sqrt{\frac{   1 - \left(\frac{c/2 -  u \sin (\pi \xi_1)}{A\sin(\pi \xi_2)} \right)^2 }{1 - u^2}} du }{\int_{-1}^1 \int_{-1}^1 \frac{du \, dv}{\sqrt{(1 - u^2)(1-v^2)}} \mathds{1}_{u \sin (\pi \xi_1) +A v \sin (\pi \xi_2) > c/2 } }. 
$$

The expression given for $B_1$ in the case $c = 0$ is directly deduced after observing that the numerator
is an even function of $u$ in this case, and that the denominator is equal to $\pi^2/2$, i.e. 
half of the integral without the indicator function, by symmetry. The expression given 
for $B_2$ for $c = 0$ is obtained in the same way, after an additional linear change of variable.

\end{proof}
\color{black}

\subsection{Parallel result in the continuous setting}
We can adapt the analysis above to the case of a continuous signal, given by a function 
$f$ from $\mathbb{R}$ to $\mathbb{R}$ such that 
$$f(t) = \cos(2 \pi t + \varphi_1) 
+ A \cos(2 \pi \xi t + \varphi_2) $$
for some $\xi > 0$, 
$\varphi_1, \varphi_2 \in \mathbb{R}$. 
Notice that discretizing this signal by taking its values at the multiples of $\varepsilon > 0$ gives the same signal as above, for $\xi_1 = \varepsilon$ and $\xi_2 = \xi \varepsilon$. 
This hinged points in the continuous setting correspond to the points where $f'$ is positive.
For a given $s \in \mathbb{R}$, 
$T > 0$, 
the average of $f(t + s)$
on the hinge points $t$ in $[-T,T]$ is given by the ratio 
$$\frac{\int_{-T}^T f(t+s) \mathds{1}_{f'(t) > 0} dt} { \int_{-T}^T  \mathds{1}_{f'(t) > 0} dt}.$$
The numerator can be written as 
$$ 
2T \mathbb{E}_{\tau} \left[ (\cos (2 \pi (\tau + s) + \varphi_1)
+ A \cos (2 \pi \xi (\tau + s) + \varphi_2 )) 
\mathds{1}_{ \sin( 2 \pi \tau + \varphi_1) + A \xi \sin(2 \pi \xi \tau + \varphi_2) < 0} \right]\,,
$$
where $\mathbb{E}_{\tau}$ denotes the expectation under the uniform probability distribution $\mathbb{P}_{\tau}$ on $[-T, T]$, and $\tau$ is a random variable following this distribution. 
The denominator is 

$$2T \mathbb{P}_{\tau} \left(
 \sin( 2 \pi \tau + \varphi_1) + A \xi \sin(2 \pi \xi \tau + \varphi_2) < 0\right),$$
 and then the ratio is 
$$\frac{\mathbb{E}_{\tau} \left[ (\cos (U_T + 2 \pi s)
+ A \cos (V_T + 2 \pi \xi s) 
\mathds{1}_{ \sin U_T + A \xi \sin V_T < 0} \right]}{\mathbb{P}_{\tau} \left(\sin U_T + A \xi \sin V_T < 0 \right) } $$
where 
$$U_T = 2 \pi \tau + \varphi_1, \; V_T = 2 \pi \xi \tau + \varphi_2.$$
If $\xi$ is irrational, the expectation 
converges to 
$$\mathbb{E}_{U,V} [(\cos (U + 2 \pi s) + A \cos (V+ 2 \pi \xi s)) \mathds{1}_{\sin U + A \xi \sin V < 0}]$$
when $T \rightarrow \infty$, $U$ and $V$ being independent uniform variables on $[0, 2 \pi)$, 
$\mathbb{E}_{U,V}$ being the expectation under the uniform distribution $\mathbb{P}_{U,V}$ on $[0, 2 \pi)^2$. 
Indeed, $(U_T, V_T)$ modulo $2 \pi$ converges to $(U,V)$ in distribution: since $\xi$ is irrational, for $(a,b) \in \mathbb{Z}^2 \backslash \{(0,0)\}$, 
$a + b \xi \neq 0$, and then 
\begin{align*}
\mathbb{E}_{\tau}[ e^{i (a U_T + b V_T)}]
& = \frac{1}{2T} e^{i (a \varphi_1 + b \varphi_2)}
\int_{-T}^T e^{2 i \pi (a + b \xi) t} dt
\\ & = \frac{1}{2T} e^{i (a \varphi_1 + b \varphi_2)} \left( \frac{e^{2 i \pi (a + b \xi) T}
-e^{-2 i \pi (a + b \xi) T} }{2 i \pi (a + b \xi)} \right)
\end{align*}
is bounded by $1/(2T \pi |a + b \xi|)$, which tends to zero when $T \rightarrow \infty$. 
We also have 
$$\mathbb{P}_{\tau} \left(\sin U_T + A \xi \sin V_T < 0 \right)  \underset{T \rightarrow \infty}{\longrightarrow} \mathbb{P}_{U,V} \left(\sin U  + A \xi \sin V  < 0 \right) = \frac{1}{2}.$$
Hence, 

$$\frac{\int_{-T}^T f(t+s) \mathds{1}_{f'(t) > 0} dt} { \int_{-T}^T  \mathds{1}_{f'(t) > 0} dt}
\underset{T \rightarrow \infty}{\longrightarrow} 
2 \, \mathbb{E}_{U,V} [(\cos (U + 2 \pi s) + A \cos (V+ 2 \pi \xi s)) \mathds{1}_{\sin U + A \xi \sin V < 0}]. 
$$

We get, expanding the cosines and noticing that 
$$\mathbb{E}_{U,V} [ (\cos U) \mathds{1}_{\sin U + A \xi \sin V < 0}] = \mathbb{E}_{U,V} [ (\cos V) \mathds{1}_{\sin U + A \xi \sin V < 0}] =0  $$
 by the symmetry $U \mapsto \pi - U$, $V \mapsto \pi - V$, we get
 \begin{align}
 B_1 \sin (2 \pi s) + B_2 \sin (2 \pi \xi s) 
 \end{align}
 for 
 $$B_1 = - 2 \mathbb{E}_{U,V} [ (\sin U) \mathds{1}_{\sin U + A \xi \sin V < 0}] =  2 \mathbb{E}_{U,V} [ (\sin U) \mathds{1}_{\sin U + A \xi \sin V >  0}] $$
 and 
 $$ B_2 = 2 A \mathbb{E}_{U,V} [ (\sin V) \mathds{1}_{\sin U + A \xi \sin V >  0}].$$

The end of the computation is then exactly similar to the discrete setting, and we again get 
\begin{align*}
B_1 = \frac{4}{\pi^2} \int_0^{\min (1, 1/(A \xi))}
\sqrt{ \frac{ 1 - A^2 \xi^2 u^2}{1 - u^2}} du.
\end{align*}
and 
 
\begin{align*}
B_2 = \frac{4 A^2 \xi}{\pi^2} \int_0^{\min (1, 1/(A \xi))}
\sqrt{ \frac{ 1 -  u^2}{1 - A^2 \xi^2  u^2}} du.
\end{align*}
This computation can be compared to 
the computation corresponding to the same signal, descretized by considering its values at multiples of $\varepsilon > 0$. We get $\xi_1 = \varepsilon$, $\xi_2 = \xi \varepsilon$: these frequencies are in $(0,1)$ if $\varepsilon < \min(1,1/\xi)$. The value of $\xi$ involved in the computation of $B_1$ and $B_2$ in the discrete case is 
$$\frac{\sin (\pi \xi_2)}{\sin (\pi \xi_1)} = 
\frac{\sin ( \pi \xi \varepsilon)}{\sin (\pi \varepsilon)},$$
which tends to $\xi$ when $\varepsilon$ tends to zero. A shift of $\ell$ in the discrete setting 
corresponds to a shift of $s = \ell \varepsilon$ in the continuous setting, and 
$$\sin(2 \pi s) = \sin (2 \pi \ell \varepsilon)
=\sin (2 \pi \ell \xi_1), $$
$$\sin(2 \pi \xi s) = \sin (2 \pi \ell \xi \varepsilon)
=\sin (2 \pi \ell \xi_2),$$
whereas the coefficients in front of $B_1$ and $B_2$ in the discrete setting are
$\sin ( \pi (2 \ell + 1) \xi_1)$ and $\sin ( \pi (2  \ell + 1) \xi_2)$. The difference between the continuous and the discrete settings for these coefficients is then dominated by $\max(\xi_1, \xi_2) \leq \varepsilon (1 + \xi)$, which tends to zero with $\varepsilon$.

\begin{rem}
While we do not need this result, it is an interesting and immediate extension of the idea of Weyl's equilibrium lemma to count the number of zeros or extrema over a sufficiently long period, which might be helpful to analyze other algorithms like empirical mode decomposition \cite{Rilling2008}. Consider the two harmonic model $f(t) = \cos (2 \pi t + \varphi_1) + A \cos (2 \pi \xi t + \varphi_2)$ for $A > 0$, $\xi \in (0,1)$, $\varphi_1, \varphi_2 \in [0, 2\pi)$. We assume $\xi$ irrational to simplify the discussion. By the same argument as above,
the number of zeros of $f$ between $0$ and $T$, when $T$ is sufficiently large, can be well approximated by  
$$2 \pi T \mathbb{E}_{U,V} [ \delta_0 (\cos  U + A \cos V) | \sin U + A \xi \sin V|]\,,$$
where $\delta_0$ is the Dirac measure. Here, we abuse the notation and skip details, while the computation can be rigorously justified like that in the discrete setup. By the symmetry $V \mapsto -V$, we get  
\begin{align}
&\pi T \mathbb{E}_{U,V} [ \delta_0 (\cos  U + A \cos V) (| \sin U + A \xi \sin V| + | \sin U - A \xi \sin V|) ]\nonumber\\
=\,&2 \pi T \mathbb{E}_{U,V} [ \delta_{-A \cos V} (\cos  U ) \max( \sqrt{1 - A^2 \cos^2 V}, A \xi \sqrt{1 - \cos^2 V})]\label{remark counting zeros eq1}\,,
\end{align}
where the equality holds since $|a + b| + |a-b| = 2 \max (|a|, |b|)$ for $a, b \in \mathbb{R}$, and $|\sin U| = \sqrt{1 - \cos^2 U}  =\sqrt{1 - A^2 \cos^2 V}$. Note that $\cos U$ is never $a\in \mathbb{R}$ if $|a| > 1$, and for $|a| < 1$, $\cos U$ takes the value $a$ twice per period of $2 \pi$. Approximating the Dirac measure and using change of variable formula, we can derive 
\begin{equation}\label{remark counting zeros eq2}
\mathbb{E}_{U,V}[ \delta_a (\cos U)] = \frac{1}{\pi \sqrt{1 - a^2}}\,.
\end{equation}
With \eqref{remark counting zeros eq2} and the tower principle, when conditioning on $V$, \eqref{remark counting zeros eq1} becomes
$$
2 T \mathbb{E}_{V} \left[ \max \left(1, A \xi \sqrt{\frac{1 - \cos^2 V}{ 1 - A^2 \cos^2 V}} \right) \mathds{1}_{A |\cos V| < 1} \right].
$$
Taking into account the density of the distribution of  $|\cos V|$ where $V$ is uniform on $[0, 2 \pi)$, we get 
$$
\frac{4 T}{\pi} \int_0^{\min (1, 1/A)}  \max \left( \frac{1}{\sqrt{1 - u^2}}, \frac{A \xi}{\sqrt{1 - A^2 u^2}} \right) du \,.
$$
When $A \leq 1$, the maximum is obtained  with the first expression, which gives 
$$\frac{4 T}{ \pi} \int_0^1 \frac{du}{\sqrt{1 - u^2}} = 2T, $$
i.e. an average of two zeros per unit of time, i.e.  per period of $\cos (2 \pi t + \varphi_1)$. When $A \xi \geq 1$, and then necessarily $A \geq 1$, the second expression gives the maximum, which gives 
$$\frac{4 T}{ \pi} \int_0^{1/A} A \xi \frac{du}{\sqrt{1 - A^2 u^2}}=\frac{4 T}{ \pi} \int_0^{1} A \xi \frac{dv/A}{\sqrt{1 - v^2}} = 2 T \xi, $$
i.e. two zeros per period of $\cos (2 \pi \xi t + \varphi_2)$. In the intermediate regime $1 \leq A \leq 1/\xi$, we have to consider the sign of 
$$1 - A^2 u^2 - A^2 \xi^2 (1 - u^2)= 1 - A^2 \xi^2 - A^2 u^2 (1 - \xi^2),$$
decreasing in $u$, vanishing at $A^{-1}\sqrt{(1 - A^2 \xi^2) / ( 1- \xi^2)}$. 
Hence, the estimate of the number of zeros of
$$\frac{4 T}{\pi} \left(\int_0^{A^{-1}\sqrt{(1 - A^2 \xi^2) / ( 1- \xi^2)}}\frac{du}{\sqrt{1 - u^2}} + \int_{A^{-1}\sqrt{(1 - A^2 \xi^2) / ( 1- \xi^2)}}^{A^{-1}} \frac{A \xi du}{\sqrt{1 - A^2 u^2}} \right)$$
i.e. 
$$\frac{4 T}{\pi} \left(\int_0^{A^{-1}\sqrt{(1 - A^2 \xi^2) / ( 1- \xi^2)}}\frac{du}{\sqrt{1 - u^2}} + \xi \int_{\sqrt{(1 - A^2 \xi^2) / ( 1- \xi^2)}}^{1} \frac{  dv}{\sqrt{1 - v^2}} \right)$$
which is
$$\frac{4 T}{ \pi} \left( \arcsin (A^{-1}\sqrt{(1 - A^2 \xi^2) / ( 1- \xi^2)}) + \xi \arccos (\sqrt{(1 - A^2 \xi^2) / ( 1- \xi^2)}) \right)$$
For the number of extremas of $f$, we consider the zeros of $f'$ and we get the same results with $A \xi$ replacing $A$. 
\end{rem}

\subsection{Numerical simulation}
We compare the PRSA output with the theoretical prediction on simulated deterministic signals that follow the two-harmonics model. Specifically, we generate $8\times 10^6$ samples at a sampling rate of $200$ Hz from $f(t)=\cos(2\pi t)+A\cos(2\pi\xi t+\phi)$, where $A>0$, $\xi\in (0,1)$ and $\phi\in[0,2\pi)$. We set $L=1000$, corresponding to a total duration of 10 seconds for each $z_{n,L}$. Figure~\ref{fig:PRSA1deterministic} compares the empirical PRSA output $z_{n,L}$ with the theoretical prediction across different values of $A$ and $\xi$, {consistent with the theoretical analysis.

Figure~\ref{fig:PRSA1deterministic2} illustrates the dependence of the PRSA output $z_{n,L}$ on the parameter $c$. We generate $8\times 10^6$ samples at a sampling rate of $10$ Hz from $f(t)=\cos(2\pi t)+0.7\cos(2\pi\times 0.2 t+\phi)$, where $\phi\in[0,2\pi)$ is randomly selected. We set $L=1000$, corresponding to a total duration of 200 seconds for each $z_{n,L}$. The results clearly demonstrate that the PRSA output varies with $c$.}

We emphasize that, even in this simple and noise-free two-harmonic model, the analysis of quasi-periodicity could be tricky. The PRSA output does not fully reflect the true structure of the underlying two-harmonic signal. Specifically, the quasi-periodic behavior of PRSA output can be different from that of the input; for example, the relative phase of two harmonic components is gone, and the relative amplitude of two harmonics is perturbed. This observation reinforces the importance of critically interpreting PRSA outputs, particularly when dealing with complex or nonstationary dynamics.

\begin{figure}[!hbt]
\centering
\includegraphics[trim=0 0 0 0, clip, width=\textwidth]{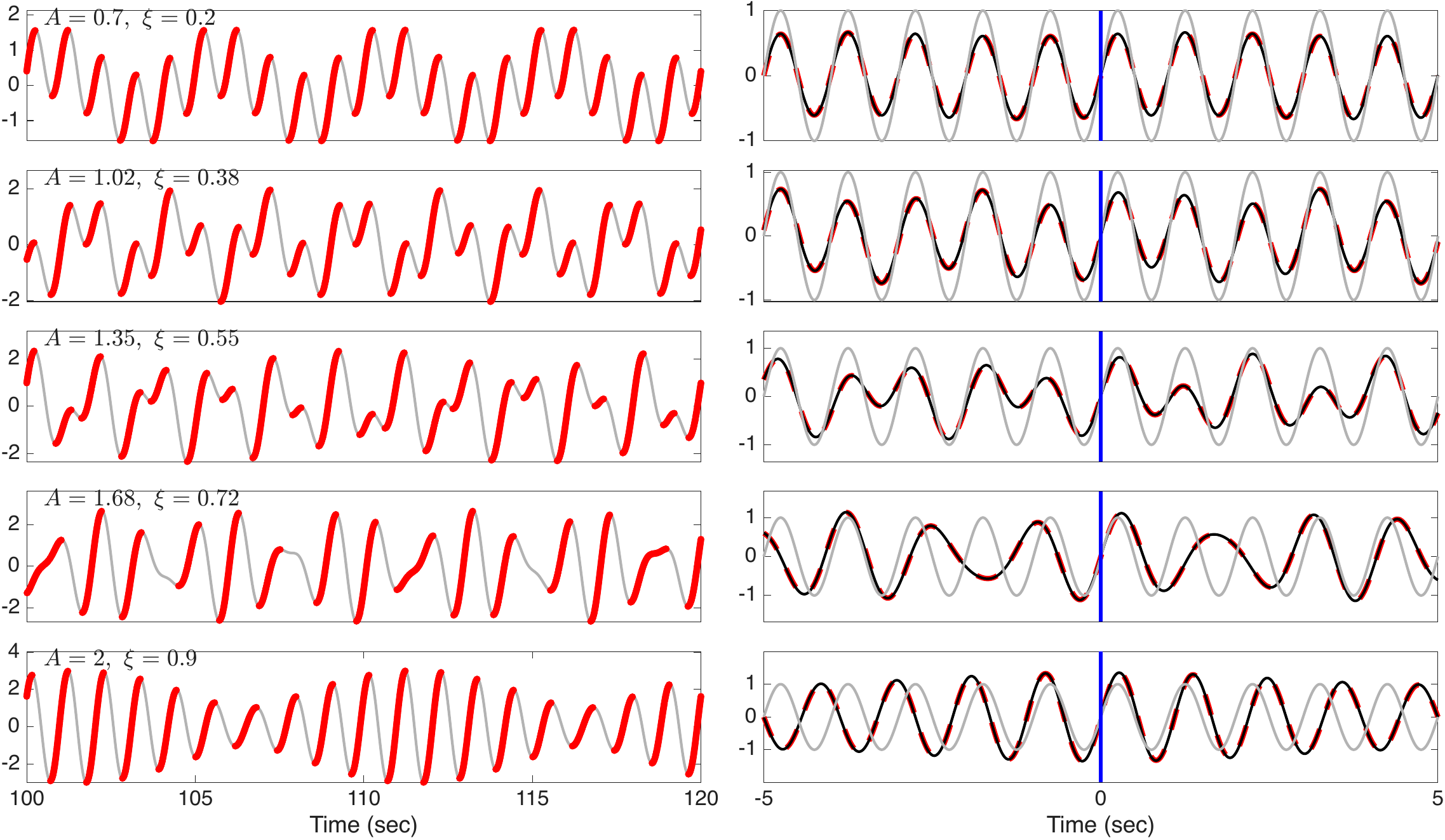} 
\caption{\label{fig:PRSA1deterministic} Illustration of PRSA with $c=0$ of deterministic signals sampled at 200 Hz  satisfying the two harmonics model with different $(A,\xi)$ pairs. Left column: a 20 seconds segment of the signal is shown in gray, with the hinge points marked in red. The $A$ and $\xi$ are shown. Right column: the low frequency component is shown in gray, the PRSA result, $z_{n,L}$ is shown in black, and the theoretically predicted output is superimposed as the dashed red curve. The vertical blue line indicated the middle point of $z_{n,L}$.}
\end{figure}

\begin{figure}[!hbt]
\centering
\includegraphics[trim=0 0 0 0, clip, width=\textwidth]{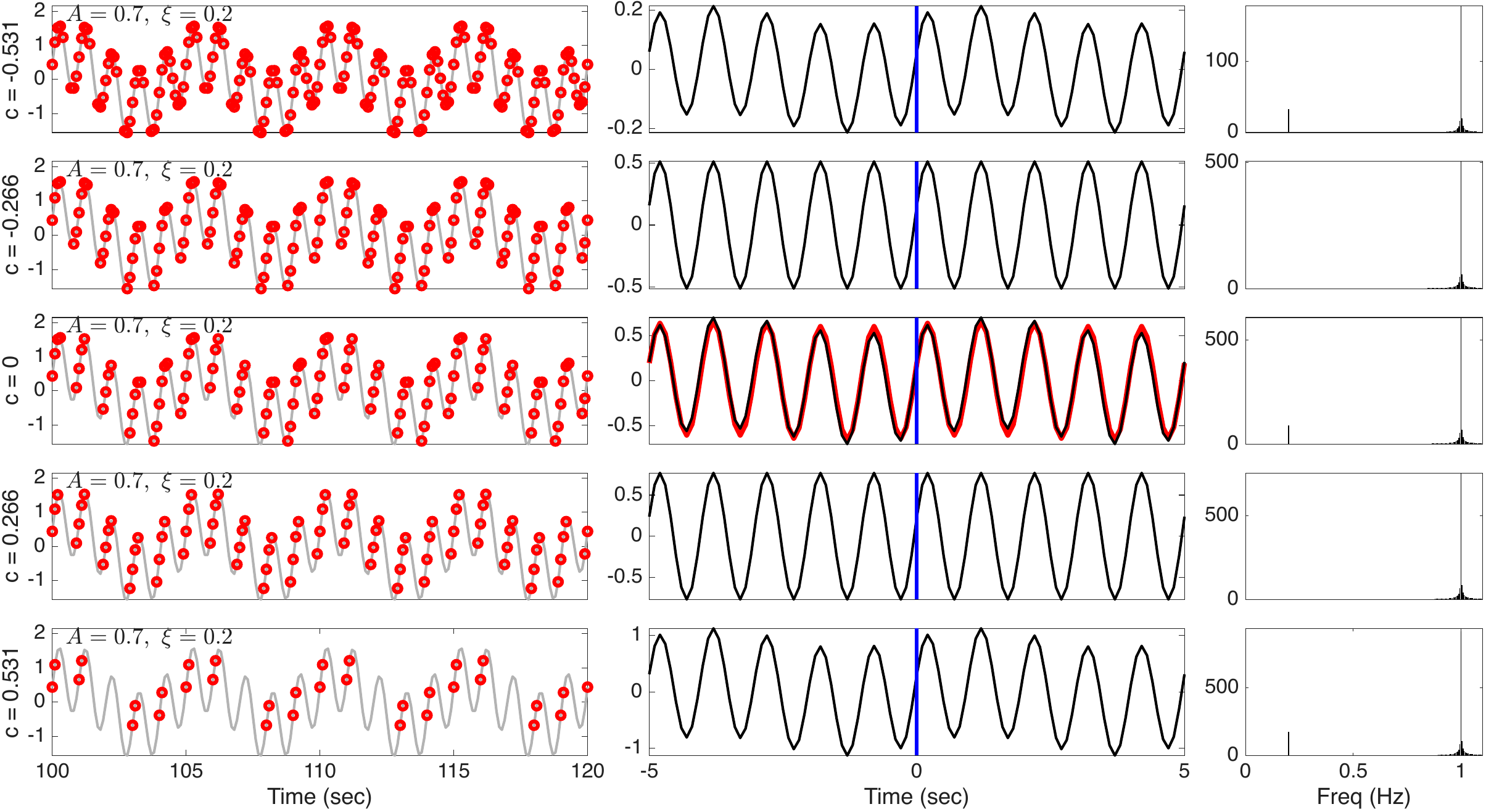} 
\caption{\label{fig:PRSA1deterministic2} Illustration of PRSA with different $c\in \mathbb{R}$ of deterministic signals $f(t)=\cos(2\pi t)+0.7\cos(2\pi\times 0.2 t+\phi)$, where $\phi\in [0,2\pi)$, sampled at 10 Hz. Left column: a 20 seconds segment of the signal is shown in gray, with the hinge points marked in red. Middle column: the PRSA result, $z_{n,L}$ is shown in black, and the theoretically predicted output is superimposed as the dashed red curve. The vertical blue line indicated the middle point of $z_{n,L}$. Right column: the magnitude of the Fourier transform of $z_{n,L}$ to indicate how the PRSA output depends on $c$.}
\end{figure}

\section{PRSA for Stochastic Processes}\label{section PRSA sto}

\subsection{Preliminaries}

To analyze the PRSA output in a stochastic setting, we begin by reformulating the average cycle \( z_{n,L} \) for a general time series \( (x_i)_{i \in \mathbb{Z}} \). This reformulation will serve as the foundation for subsequent probabilistic analysis.

\begin{lem}\label{lemma 4.1}
Let \( (x_i)_{i \in \mathbb{Z}} \) be a real-valued time series, and define the finite-difference sequence \( w_i := x_i - x_{i-1} \). Let the set of hinge points be determined by a fixed threshold \( c \in \mathbb{R} \), i.e., hinge points are indices \( i \) such that \( w_i > c \).

Consider the PRSA-averaged cycle \( z_{n,L} = (_{n,\ell})_{-L \leq \ell \leq L} \in \mathbb{R}^{2L+1} \). Then for each \( \ell \in [-L, L] \), we have
\[
z_{n,L}(\ell) = x_0 + \sum_{p=1}^{n+\ell} w_p \cdot \frac{H(p - \ell - 1, n)}{H(-n - 1, n)} 
- \sum_{p = -n+\ell+1}^{0} w_p \cdot \frac{H(-n - 1, p - \ell - 1)}{H(-n - 1, n)},
\]
where \( H(a, b) \) denotes the number of hinge points with indices in the interval \( (a, b] \).
\end{lem}

\begin{proof}

We have: 
$$z_{n,L}(\ell) = \frac{1}{\sum_{m = -n}^n \mathds{1}_{\{w_m > c\}}} \sum_{m = -n}^n x_{m+\ell} \mathds{1}_{\{w_m > c\}}.
$$
This expression reflects a weighted average of the values \( x_{m+\ell} \), where \( m \) runs over indices whose increment \( w_m \) exceeds the threshold \( c \).

We now express each \( x_{m+\ell} \) in terms of the base point \( x_0 \) and the increments \( w_p \). When \( m + \ell \geq 0 \), we write
\[
x_{m+\ell} = x_0 + \sum_{p=1}^{m+\ell} w_p.
\]
When \( m + \ell < 0 \), we write
\[
x_{m+\ell} = x_0 - \sum_{p = m+\ell+1}^{0} w_p.
\]

We split the full sum into two parts: the positive index part (where \( m + \ell \geq 0 \)) and the negative index part (where \( m + \ell < 0 \)).

{\bf Positive part:}
We have 
\[
\sum_{-n \leq m \leq n, m+\ell \geq 0} x_{m+\ell} \mathds{1}_{\{w_m > c\}} = \sum_{-n \leq m \leq n, m+\ell \geq 0} \left(x_0 + \sum_{p=1}^{m+\ell} w_p \right) \mathds{1}_{\{w_m > c\}}.
\]
We separate the constant term and switch the order of summation in the second term:
\[
\sum_{-n \leq m \leq n, m+\ell \geq 0} x_{m+\ell} \mathds{1}_{\{w_m > c\}} = x_0 \sum_{-n \leq m \leq n, m + \ell \geq 0} \mathds{1}_{\{w_m > c\}} + \sum_{p = 1}^{n+\ell} w_p \sum_{\substack{m = -n \\ m + \ell \geq p}}^{n} \mathds{1}_{\{w_m > c\}}.
\]
Observe that the inner sum counts how many hinge points \( w_m > c \) correspond to \( m \in [p - \ell, n] \). Thus, this count equals \( H(p - \ell - 1, n) \).

{\bf Negative part:}
For \( m \) such that \( m + \ell < 0 \), we have
\[
x_{m+\ell} = x_0 - \sum_{p = m+\ell+1}^{0} w_p.
\]
Proceeding as before, we obtain
\begin{align*}
\sum_{-n \leq m \leq n, m + \ell < 0} x_{m+\ell} \mathds{1}_{\{w_m > c\}} & = x_0 \sum_{-n \leq m \leq n, m + \ell < 0 } \mathds{1}_{\{w_m > c\}} \\ & - \sum_{p = -n+\ell+1}^{0} w_p \sum_{\substack{m = -n \\ m + \ell < p}}^{n} \mathds{1}_{\{w_m > c\}}.
\end{align*}
The inner sum counts the hinge points in the interval \( (-n - 1, p - \ell - 1] \), which equals \( H(-n - 1, p - \ell - 1) \).

{\bf Putting everything together:}
Adding both contributions and normalizing by the total number of hinge points \( H(-n - 1, n) \), we obtain
\[
z_{n,L} (\ell)= x_0 + \sum_{p = 1}^{n + \ell} w_p \cdot \frac{H(p - \ell - 1, n)}{H(-n - 1, n)} - \sum_{p = -n + \ell + 1}^{0} w_p \cdot \frac{H(-n - 1, p - \ell - 1)}{H(-n - 1, n)},
\]
which is the desired result.
\end{proof}

\begin{rem}
The expression in Lemma \ref{lemma 4.1} rewrites each coordinate \( z_{n,L}(\ell) \) of the PRSA-averaged signal as a weighted sum of the increments \( w_p = x_p - x_{p-1} \), centered at \( x_0 \). These weights are determined by the number of hinge points occurring in specific intervals, reflecting how the hinge-point structure influences the contribution of the increments. This reformulation is essential for our subsequent analysis: it enables us to apply probabilistic tools to study the asymptotic behavior of PRSA when the input signal is modeled as a stochastic process.
\end{rem}

%We need the following technical bound: 
The main technical challenge in analyzing the variance and higher-order moments of the PRSA output is to understand how the dependence between distant components of the time series propagates through nonlinear functionals, such as \( x_{m+\ell} \mathds{1}_{\{w_m > c\}} \). The following result provides a quantitative bound on the covariance between functionals of jointly Gaussian vectors in terms of their smoothness and coordinate-level correlations.

\begin{prop}\label{proposition approximation bound}
Let \( X_1 \in \mathbb{R}^{d_1} \) and \( X_2 \in \mathbb{R}^{d_2} \) be centered Gaussian random vectors such that the joint vector \( (X_1, X_2) \) is also Gaussian. Let \( \Phi_1 : \mathbb{R}^{d_1} \to \mathbb{R} \) and \( \Phi_2 : \mathbb{R}^{d_2} \to \mathbb{R} \) be continuously differentiable functions.

Define:
\begin{itemize}
  \item \( M := \max\{ \operatorname{Var}((X_1)_j), \operatorname{Var}((X_2)_k) : 1 \leq j \leq d_1,\, 1 \leq k \leq d_2 \} \),
  \item \( \mu := \max\{ |\operatorname{Cov}((X_1)_j, (X_2)_k)| : 1 \leq j \leq d_1,\, 1 \leq k \leq d_2 \} \),
  \item \( N(\Phi_s) := |\Phi_s(0)| + \sup_{x \in \mathbb{R}^{d_s}} \| \nabla \Phi_s(x) \| \), for \( s = 1,2 \).
\end{itemize}

Then the covariance between \( \Phi_1(X_1) \) and \( \Phi_2(X_2) \) satisfies
\[
\left| \mathbb{E}[\Phi_1(X_1)\Phi_2(X_2)] - \mathbb{E}[\Phi_1(X_1)]\mathbb{E}[\Phi_2(X_2)] \right|
\leq 7  d_1 d_2\, N(\Phi_1)\, N(\Phi_2)\, \sqrt{ \mu(1 + M) }.
\]
\end{prop}

\begin{proof}
We define
\[
C_{jk} := \operatorname{Cov}((X_1)_j, (X_2)_k), \quad 1 \leq j \leq d_1,\ 1 \leq k \leq d_2,
\]
which implies \( \mu = \max_{j,k} |C_{jk}| \). 

We construct a block matrix
\[
A :=
\begin{pmatrix}
 \mu d_1 I_{d_1} & C \\
C^T & \mu  d_2 I_{d_2}
\end{pmatrix},
\]
and claim that \( A \) is positive semidefinite. To see this, take arbitrary vectors \( Y_1 \in \mathbb{R}^{d_1} \), \( Y_2 \in \mathbb{R}^{d_2} \). Then:
\begin{align*}
\begin{pmatrix} Y_1^T & Y_2^T \end{pmatrix}
A
\begin{pmatrix} Y_1 \\ Y_2 \end{pmatrix}
&= \mu  d_1 \|Y_1\|^2 + \mu d_2 \|Y_2\|^2 + 2 Y_1^T C Y_2 \\
&\geq \mu \left( d_1 \|Y_1\|^2 +  d_2 \|Y_2\|^2 - 2 \sqrt{d_1 d_2} \|Y_1\| \|Y_2\| \right) \\
&\geq 0,
\end{align*}
since
$$|Y_1^T C Y_2| 
= \left| \sum_{j=1}^{d_1} \sum_{k=1}^{d_2}
(Y_1)_j (Y_2)_k C_{jk} \right| \leq 
\mu  \sum_{j=1}^{d_1} 
\sum_{k=1}^{d_2} |(Y_1)_j || (Y_2)_k| 
= \mu \sum_{j=1}^{d_1}  |(Y_1)_j |
\sum_{k=1}^{d_2}| (Y_2)_k|
$$
and then,
by Cauchy-Schwarz inequality, 
$$|Y_1^T C Y_2|  \leq 
\mu \sqrt{d_1 d_2} \|Y_1\| \|Y_2\|. 
$$

Now, let \( N_1 \sim \mathcal{N}(0, I_{d_1}) \), \( N_2 \sim \mathcal{N}(0, I_{d_2}) \) be independent standard Gaussian vectors, also independent of \( (X_1, X_2) \). Define the ``regularized" variables:
\[
\widetilde{X}_1 := X_1 + \sqrt{d_1 \mu}\, N_1, \quad \widetilde{X}_2 := X_2 + \sqrt{d_2 \mu}\, N_2.
\]
The joint covariance matrix of \( (\widetilde{X}_1, \widetilde{X}_2) \) is:
\begin{align*}
\operatorname{Cov}(\widetilde{X}_1, \widetilde{X}_2) &  =
\begin{pmatrix}
\operatorname{Cov}(X_1) + d_1 \mu I_{d_1} & C \\
C^T & \operatorname{Cov}(X_2) + d_2 \mu I_{d_2}
\end{pmatrix}
\\ & = \begin{pmatrix}
\operatorname{Cov}(X_1)  & 0 \\
0 & \operatorname{Cov}(X_2) 
\end{pmatrix} + A.
\end{align*}

Let \( Z \sim \mathcal{N}(0, A) \), \(X'_1\)
and \(X'_2\) be independent, 
\(X'_1\) being distributed as \(X_1\)
and \(X'_2\) being distributed as \(X_2\). Comparing the covariance matrices, we deduce that the pair \( (X_1',X_2') + Z \) has the same law as \( (\widetilde{X}_1, \widetilde{X}_2) \). If $Z_1$ is the vector of the $d_1$ first components of $Z$ and $Z_2$ is the vector of the $d_2$ last components of $Z$,
we get, from this equality in law, 
\begin{align}
 & \mathbb{E}[ \Phi_1(X_1) \Phi_2(X_2)] 
- \mathbb{E} [ \Phi_1 (X_1)] \mathbb{E} [\Phi_2(X_2)] \nonumber \\
=\,& 
\mathbb{E}[ \Phi_1(X_1) \Phi_2(X_2)] 
- \mathbb{E} [ \Phi_1 (X'_1)\Phi_2(X'_2)]
\nonumber \\ 
=\,& \mathbb{E}[ \Phi_1(X_1) \Phi_2(X_2)] 
- \mathbb{E}[ \Phi_1(\widetilde{X}_1)
 \Phi_2(\widetilde{X}_2)] 
 \nonumber \\ & - \left( \mathbb{E} [ \Phi_1 (X'_1)\Phi_2(X'_2)]
 - \mathbb{E} [ \Phi_1 (X'_1 + Z_1)\Phi_2(X'_2 + Z_2)] \right)  \label{decorrelationPhi1Phi2}
\end{align}

We now bound each of the two last differences
of expectations using regularity properties of 
\( \Phi_1, \Phi_2 \). 
For the first difference, we have 
\begin{align*}
|\mathbb{E}[\Phi_1(X_1)\Phi_2(X_2)] - \mathbb{E}[\Phi_1(\widetilde{X}_1)\Phi_2(\widetilde{X}_2)]|
&\leq \mathbb{E}\left[ \left| \Phi_1(X_1) - \Phi_1(\widetilde{X}_1) \right| |\Phi_2(X_2)| \right] \\
&\quad + \mathbb{E}\left[ |\Phi_1(\widetilde{X}_1)| \, \left| \Phi_2(X_2) - \Phi_2(\widetilde{X}_2) \right| \right].
\end{align*}
We have, from the definition of 
$N(\Phi_1)$ and $N(\Phi_2)$, 
$$ |\Phi_1 (X_1) - \Phi_1 (\widetilde{X}_1)|
 \leq N (\Phi_1) \|X_1 - \widetilde{X}_1\|,$$
$$|\Phi_2 (X_2) |
\leq |\Phi_2 (0) |
+ \|X_2\| \sup_{x \in \mathbb{R}^{d_2}} \|\nabla \Phi_2 (x)\|
\leq  N(\Phi_2) (1 + \|X_2\|),
$$
and then, using Cauchy-Schwarz inequality,  
\begin{align*} &  \mathbb{E}\left[ \left| \Phi_1(X_1) - \Phi_1(\widetilde{X}_1) \right| |\Phi_2(X_2)| \right] 
\\ \leq \,& N(\Phi_1) N(\Phi_2) 
\left(\mathbb{E} [ \|X_1 - \widetilde{X}_1\|^2] \right) ^{1/2}
\left(\mathbb{E} [ ( 1 + \|X_2\|)^2] \right)^{1/2}
 \\  \leq \,& N(\Phi_1) N(\Phi_2) 
\left(\mathbb{E} [  d_1 \mu \|N_1\|^2] \right) ^{1/2}
\left(\mathbb{E} [ 2 + 2 \|X_2\|^2] \right)^{1/2}\,.
\end{align*}
We have 
$$\mathbb{E} [ \|N_1\|^2] = d_1, \; 
\mathbb{E} [ \|X_2\|^2] \leq d_2 M,$$
and then 
\begin{align} \mathbb{E}\left[ \left| \Phi_1(X_1) - \Phi_1(\widetilde{X}_1) \right| |\Phi_2(X_2)| \right] & \leq  N(\Phi_1) N (\Phi_2) \sqrt{d_1 \mu} \sqrt{d_1}  ( 2 + 2 d_2 M)^{1/2} 
\nonumber \\ & \leq N(\Phi_1) N(\Phi_2) d_1 \sqrt{2d_2  \mu(1 + M)} 
\nonumber \\ & \leq N(\Phi_1) N(\Phi_2) d_1 d_2 \sqrt{2  \mu(1 + M)}  \label{Phi1Phi2XXtilde}
\end{align}
Similarly, 
\begin{align*} &  \mathbb{E}\left[ |\Phi_1(\widetilde{X}_1)| \cdot \left| \Phi_2(X_2) - \Phi_2(\widetilde{X}_2) \right| \right]
\\ \leq \,& N(\Phi_1) N (\Phi_2) 
\left(\mathbb{E} [ d_2 \mu \|N_2\|^2] \right)^{1/2}
\left(\mathbb{E} [ 2 + 2 \|\widetilde{X}_1\|^2] \right)^{1/2}\,,
\end{align*}
where 
$$\mathbb{E} [ \|N_2\|^2] = d_2, \ \  
\mathbb{E} [ \|\widetilde{X}_1\|^2] \leq d_1 ( M
+ d_1 \mu) \leq d_1 (1 + d_1) M,$$
the last inequality being due to the fact that 
$\mu \leq M$, because of Cauchy-Schwarz inequality and the definitions of $M$ and $\mu$. 
We deduce
\begin{align*} 
&  \mathbb{E}\left[ |\Phi_1(\widetilde{X}_1)| \cdot \left| \Phi_2(X_2) - \Phi_2(\widetilde{X}_2) \right| \right]
\\  \leq \,& N(\Phi_1) N (\Phi_2) 
\sqrt{d_2 \mu} \sqrt{d_2} (2 + 2d_1 (1 + d_1)M)^{1/2}
\\ \leq \,& N(\Phi_1) N (\Phi_2) 
\sqrt{d_2 \mu} \sqrt{d_2} (4 d_1^2 (1 + M))^{1/2} 
\\ = \,& N(\Phi_1) N(\Phi_2) d_1 d_2 \sqrt{4  \mu(1 + M)}.
\end{align*}
Adding this estimate to \eqref{Phi1Phi2XXtilde}, we deduce 
\begin{align}
&|\mathbb{E}[\Phi_1(X_1)\Phi_2(X_2)] - \mathbb{E}[\Phi_1(\widetilde{X}_1)\Phi_2(\widetilde{X}_2)]|\nonumber\\
\leq \,& (2 + \sqrt{2}) N(\Phi_1) N(\Phi_2) d_1 d_2 \sqrt{  \mu(1 + M)}.  \label{Phix1tildexx}
\end{align}
Similar computation as above gives
\begin{align*}
& |\mathbb{E}[\Phi_1(X'_1)\Phi_2(X'_2)] - \mathbb{E}[\Phi_1(X'_1 + Z_1)\Phi_2(X'_2 + Z_2)]|
\\ \leq \,& N(\Phi_1) N(\Phi_2) 
 \left(\mathbb{E} [ \|Z_1\|^2] \right)^{1/2} 
\left(\mathbb{E} [ 2 + 2 \|X'_2\|^2 ] \right)^{1/2}
\\ & + N(\Phi_1) N(\Phi_2)  \left(\mathbb{E} [ \|Z_2\|^2] \right)^{1/2} 
\left(\mathbb{E} [ 2 + 2 \|X'_1 + Z_1\|^2 ] \right)^{1/2}.
\end{align*}
Now, one checks that $X'_2$ has the same distribution as $X_2$, 
$X'_1 + Z_1$ has the same distribution as $\widetilde{X}_1$, $Z_1$ has the same distribution as $\sqrt{\mu d_1} N_1$, and $Z_2$ has the same distribution as $\sqrt{\mu d_2} N_2$. 
We deduce
\begin{align*}
& |\mathbb{E}[\Phi_1(X'_1)\Phi_2(X'_2)] - \mathbb{E}[\Phi_1(X'_1 + Z_1)\Phi_2(X'_2 + Z_2)]|
\\ \leq \,& N(\Phi_1) N(\Phi_2) 
 \left(\mathbb{E} [ \mu d_1 \|N_1\|^2] \right)^{1/2} 
\left(\mathbb{E} [ 2 + 2 \|X_2\|^2 ] \right)^{1/2}
\\ & + N(\Phi_1) N(\Phi_2)  \left(\mathbb{E} [ \mu d_2 \|N_2\|^2] \right)^{1/2} 
\left(\mathbb{E} [ 2 + 2 \|\widetilde{X}_1 \|^2 ] \right)^{1/2}.
\end{align*}
By the computation above, 
\begin{align*}
& |\mathbb{E}[\Phi_1(X'_1)\Phi_2(X'_2)] - \mathbb{E}[\Phi_1(X'_1 + Z_1)\Phi_2(X'_2 + Z_2)]|
\\ \leq  \,&  (2 + \sqrt{2}) N(\Phi_1) N(\Phi_2) d_1 d_2 \sqrt{  \mu(1 + M)}.
\end{align*}
Summing this estimate with \eqref{Phix1tildexx}
and using \eqref{decorrelationPhi1Phi2}, we deduce 
\begin{align*}
& \left| \mathbb{E}[\Phi_1(X_1)\Phi_2(X_2)] - \mathbb{E}[\Phi_1(X_1)] \mathbb{E}[\Phi_2(X_2)] \right|
\\ \leq \,& (4 + 2 \sqrt{2}) d_1 d_2\, N(\Phi_1)\, N(\Phi_2)\, \sqrt{ \mu(1 + M) }
\end{align*}
and hence the claim.
\end{proof}

\begin{rem}
Proposition \ref{proposition approximation bound} provides a quantitative bound on the covariance between two continous differentiable functionals of jointly Gaussian vectors. The key idea is that, even when two vectors \( X_1 \) and \( X_2 \) are not independent, their functionals \( \Phi_1(X_1) \) and \( \Phi_2(X_2) \) may be nearly uncorrelated if the maximal covariance \( \mu \) between coordinates of \( X_1 \) and \( X_2 \) is small.

This result is particularly useful in analyzing   dependencies in Gaussian processes. In our setting, it allows us to bound the covariance between terms of the form \( x_{m+\ell} \mathds{1}_{\{w_m > c\}} \), which arise in the numerator and denominator of the PRSA average. The dependence between such terms decays as the time indices separate, and this proposition helps us make that decay quantitative by leveraging the Gaussian structure and regularity of \( \Phi_1, \Phi_2 \).
\end{rem}

We now adjust Proposition \ref{proposition approximation bound} to the functionals arising in our analysis of the PRSA average, where each term involves a product of a signal value and a thresholded increment.
\begin{prop}[Covariance bound for PRSA terms]  \label{covarianceboundPRSA}
Let \( (x_n)_{n \in \mathbb{Z}} \) be a centered stationary Gaussian process with covariance function \( C(k) = \mathbb{E}[x_0 x_k] \), and define increments \( w_n := x_n - x_{n-1} \). 
We assume $C(0) > |C(1)|$. For integers $p \geq 1$, $a < b$,  
let $X$ be a random variable equal to the product of at most 
$p$ random factors among 
$x_n$ and $\mathds{1}_{w_n > c}$, for $n \leq a$, 
i.e. 
$$X = x_{n_1} x_{n_2} 
\dots x_{n_r}   \mathds{1}_{w_{n_{r+1} > c}}
\dots \mathds{1}_{w_{n_{r + s}} > c}$$
 for some integers $n_1, \dots, n_{r+s} \leq a$, $r +s$ being at most $p$. Similarly, let $Y$ be a random variable equal to the product of at most $p$ factors among $x_n$ and $\mathds{1}_{w_n > c}$ for $n \geq b$: 
$$Y = x_{n'_1} x_{n'_2} 
\dots x_{n'_{r'}}   \mathds{1}_{w_{n'_{r'+1} > c}}
\dots \mathds{1}_{w_{n'_{r' + s'}} > c}$$
 for $n'_1, \dots, n'_{r'+s'} \geq b$, $r' +s'$ being at most $p$.  
Then
$$|\mathbb{E} [ XY] - \mathbb{E} [X] \mathbb{E}[Y] | 
\leq K(p) (C(0)- C(1))^{-\frac{1}{4p}} (1  + C(0)^{p+1 + \frac{1}{4p} - \frac{1}{12p-2}} )\left(\sup_{\ell \geq b-a-1} |C(\ell)| \right)^{\frac{1}{12p-2}}  $$
for $K(p) > 0$ depending only on $p$. 

\end{prop}

Note that this bound is independent of the chosen $c$.

\begin{proof}
Let $\delta \in (0,1),  R > 1$. 
There are functions $g_{\delta}$ and $h_R$ from $\mathbb{R}$ to $\mathbb{R}$ satisfying the following properties: 
\begin{itemize}
\item $g_{\delta}$ takes values in $[0,1]$, is  bounded and smooth, with derivative 
dominated by $1/\delta$, and coincides with the indicator of $[0, \infty)$ outside the interval $[- \delta, \delta]$. 
\item $h_R$ is bounded, smooth, increasing,  coincides with the identity function on $[-R, R]$, tends to $2R$ at infinity and to $-2R$ at minus infinity, and satisfies $h'_R \leq 1$. 
\end{itemize}
The first step consists in 
 replacing one by one 
each factor $x_m$ by $h_R(x_m)$ and each factor $\mathds{1}_{w_m > c}$ by 
$g_{\delta} (w_m - c)$ in the expressions of $X$ and $Y$.
After each replacement, we get new random variables
which we will still name $X$ and $Y$, and we will 
study how much the quantity $\mathbb{E}[XY] - \mathbb{E}[X] \mathbb{E}[Y]$ is modified by the replacement. 
 Each replacement  
of $x_m$ by $h_R (x_m)$
changes only one among the two expectations $\mathbb{E} [ X]$ and $\mathbb{E} [ Y]$, 
the change being bounded, via H\"older inequality,  stationarity of the Gaussian process, and the fact that $|h_R(x)| \leq |x|$ because $h_R' \leq 1$,  by
\begin{align*} \mathbb{E} [ |x_{m_1}| | x_{m_2} | & 
\dots |x_{m_v}| |x_m - h_R (x_m)|]
 \leq \mathbb{E} [ |x_{m_1}| | x_{m_2} |
\dots |x_{m_v}| |x_m | \mathds{1}_{|x_m| \geq R}]
\\ & \leq  R^{-1} \mathbb{E} [ |x_{m_1}| | x_{m_2} |
\dots |x_{m_v}| |x_m |^2]
\\ & \leq R^{-1} \prod_{j=1}^v 
\left(\mathbb{E}[ |x_{m_j}|^{v+2}] \right)^{1/(v+2)}
\left(\mathbb{E} [ |x_m|^{v+2}] \right)^{2/(v+2)}
\\ & \leq R^{-1} (C(0))^{1 + v/2} 
\mathbb{E} [|\mathcal{N}|^{v+2}] \leq K_1(p) R^{-1} (1 + (C(0))^{(1 + p)/2})
\end{align*}
for some integers $m_1, \dots, m_v$, $v \leq p-1$, where $\mathcal{N}$ is a standard normal variable variable, and $K_1(p) > 0$ depends only on $p$. 
Hence, replacing one of the factors $x_m$ by $h_R(x_m)$ changes the expectation of one of the variables $X$ or $Y$ ($X$ if $m \leq a$ and $Y$ if $m\geq b$) by a quantity bounded by 
$K_1(p) R^{-1} (1 + (C(0))^{(1 + p)/2})$, while the other expectation being unchanged. Here, this unchanged expectation is bounded by 
$$
\mathbb{E} [ |x_{m_1}| \dots |x_{m_v}|] \leq K_2 (p) ( 1 + C(0)^{p/2}) 
$$
for some integers $m_1, \dots, m_v$, $v \leq p$, 
$K_2(p) > 0$ depending only on $p$. 
Hence, changing successively $x_n$ by $h_R(x_n)$ in the expression of $X$ and $Y$ changes the value of $\mathbb{E}[X] \mathbb{E}[ Y]$ by 
at most 
$$K_1(p) K_2(p) 
R^{-1} (1 + C(0)^{p/2}) (1 + C(0)^{(p+1)/2})$$
for each change. 
More explicitly, we have 
\begin{align*}
& \left| \mathbb{E} [ h_R(x_{n_1}) h_R(x_{n_2}) \dots h_R(x_{n_v})
x_{n_{v+1}} \dots x_{n_r}  \mathds{1}_{w_{n_{r+1} > c}}
\dots \mathds{1}_{w_{n_{r + s}} > c} ] \right.
\\ &  \quad \times \mathbb{E} [ h_R(x_{n'_1}) h_R(x_{n'_2}) \dots h_R(x_{n'_{v'}})
x_{n'_{v'+1}} \dots x_{n'_{r'}}  \mathds{1}_{w_{n'_{r'+1} > c}}
\dots \mathds{1}_{w_{n'_{r' + s'}} > c}   
]  
\\ & -  \mathbb{E} [ h_R(x_{n_1}) h_R(x_{n_2}) \dots h_R(x_{n_v})
h_R(x_{n_{v+1}}) x_{n_{v+2}} \dots x_{n_r}  \mathds{1}_{w_{n_{r+1} > c}}
\dots \mathds{1}_{w_{n_{r + s}} > c} ]  
\\ &  \left. \quad \times \mathbb{E} [ h_R(x_{n'_1}) h_R(x_{n'_2}) \dots h_R(x_{n'_{v'}})
x_{n'_{v'+1}} \dots x_{n'_{r'}}  \mathds{1}_{w_{n'_{r'+1} > c}}
\dots \mathds{1}_{w_{n'_{r' + s'}} > c} ]   \right|
\\ \leq \,& K_1(p) K_2(p) 
R^{-1} (1 + C(0)^{p/2}) (1 + C(0)^{(p+1)/2})
\end{align*}
for integers $r \geq 1$, $s \geq 0$, $0 \leq v \leq r-1$,  
$n_1, \dots n_{r+s} \leq a$, $r' \geq 0$, $s' \geq 0$, $0 \leq v' \leq r'$,  
$n'_1, \dots n'_{r'+s'} \geq b$, such that $r + s$ and $r' + s'$ are between $1$ and $p$. 

Similarly, the modification of $\mathbb{E}[XY]$ when 
we replace each factor $x_n$ by $h_R(x_n)$ is bounded by differences of the form
\begin{align*}
& \left| \mathbb{E} [ h_R(x_{n_1}) h_R(x_{n_2}) \dots h_R(x_{n_v})
x_{n_{v+1}} \dots x_{n_r}  \mathds{1}_{w_{n_{r+1} > c}}
\dots \mathds{1}_{w_{n_{r + s}} > c}  \right.
\\ &  \quad   \times  h_R(x_{n'_1}) h_R(x_{n'_2}) \dots h_R(x_{n'_{v'}})
x_{n'_{v'+1}} \dots x_{n'_{r'}}  \mathds{1}_{w_{n'_{r'+1} > c}}
\dots \mathds{1}_{w_{n'_{r' + s'}} > c}   
]  
\\ & -  \mathbb{E} [ h_R(x_{n_1}) h_R(x_{n_2}) \dots h_R(x_{n_v})
h_R(x_{n_{v+1}}) x_{n_{v+2}} \dots x_{n_r}  \mathds{1}_{w_{n_{r+1} > c}}
\dots \mathds{1}_{w_{n_{r + s}} > c}   
\\ &  \left. \quad\times   h_R(x_{n'_1}) h_R(x_{n'_2}) \dots h_R(x_{n'_{v'}})
x_{n'_{v'+1}} \dots x_{n'_{r'}}  \mathds{1}_{w_{n'_{r'+1} > c}}
\dots \mathds{1}_{w_{n'_{r' + s'}} > c} ]   \right|.
\end{align*}
These differences are bounded similarly as the modification of $\mathbb{E}[X]$ and $\mathbb{E}[Y]$. The only modification in the bound is the fact that $XY$ is a product of at most $2p$ factors 
of the form $x_n$ or $\mathds{1}_{w_n > c}$, instead of $p$ factors for $X$ or $Y$. 
Using H\"older inequality and stationarity as above, we get 
\begin{align*}
& \left| \mathbb{E} [ h_R(x_{n_1}) h_R(x_{n_2}) \dots h_R(x_{n_v})
x_{n_{v+1}} \dots x_{n_r}  \mathds{1}_{w_{n_{r+1} > c}}
\dots \mathds{1}_{x_{n_{r + s}} > c}  \right.
\\ &    \quad \times  h_R(x_{n'_1}) h_R(x_{n'_2}) \dots h_R(x_{n'_{v'}})
x_{n'_{v'+1}} \dots x_{n'_{r'}}  \mathds{1}_{w_{n'_{r'+1} > c}}
\dots \mathds{1}_{w_{n'_{r' + s'}} > c}   
]  
\\ & -  \mathbb{E} [ h_R(x_{n_1}) h_R(x_{n_2}) \dots h_R(x_{n_v})
h_R(x_{n_{v+1}}) x_{n_{v+2}} \dots x_{n_r}  \mathds{1}_{w_{n_{r+1} > c}}
\dots \mathds{1}_{w_{n_{r + s}} > c}   
\\ &  \left.\quad \times   h_R(x_{n'_1}) h_R(x_{n'_2}) \dots h_R(x_{n'_{v'}})
x_{n'_{v'+1}} \dots x_{n'_{r'}}  \mathds{1}_{w_{n'_{r'+1} > c}}
\dots \mathds{1}_{w_{n'_{r' + s'}} > c} ]   \right|
\\ \leq \,& K_3 (p) R^{-1} (1 + C(0)^{1/2 + p})
\end{align*}
for $K_3(p) > 0$ depending only on $p$. 

 Adding the bounds above on the modification of 
 $\mathbb{E}[X] \mathbb{E}[Y]$ and $\mathbb{E}[XY]$, we deduce that $\mathbb{E}[XY] - \mathbb{E}[X] \mathbb{E}[Y]$ 
is modified by at most $K_4(p)R^{-1}(1 + C(0)^{1/2 + p})$
for $K_4(p) > 0$ depending only on $p$, for each replacement of a factor $x_n$
by $h_R(x_n)$.  

Once all factors $x_n$ have been replaced by $h_R(x_n)$, we consider the effect of 
successive replacements of the factors $\mathds{1}_{w_n > c}$ by $g_{\delta}(w_n - c)$.
We have to bound expressions of the form
\begin{align*}
& \left| \mathbb{E} [ h_R(x_{n_1}) \dots h_R(x_{n_r})
  g_{\delta} (w_{n_{r+1}} - c) 
  \dots g_{\delta} (w_{n_{r + v}} - c) 
  \mathds{1}_{w_{n_{r + v + 1}} > c}
\dots \mathds{1}_{w_{n_{r + s}} > c} ] \right.
\\ & \quad \times \mathbb{E} [ h_R(x_{n'_1})   \dots h_R(x_{n'_{r'}})
  g_{\delta} (w_{n'_{r'+1}} - c) 
  \dots g_{\delta} (w_{n'_{r' + v'}} - c) 
  \mathds{1}_{w_{n'_{r' + v' + 1}} > c}
\dots \mathds{1}_{w_{n'_{r' + s'}} > c}  
]  
\\ & -   \mathbb{E} [ h_R(x_{n_1}) \dots h_R(x_{n_r})
  g_{\delta} (w_{n_{r+1}} - c) 
  \dots g_{\delta} (w_{n_{r + v}} - c) 
  g_{\delta} (w_{n_{r + v+1}} - c) 
\dots \mathds{1}_{w_{n_{r + s}} > c} ]  
\\ & \left.\quad \times \mathbb{E} [ h_R(x_{n'_1})   \dots h_R(x_{n'_{r'}})
  g_{\delta} (w_{n'_{r'+1}} - c) 
  \dots g_{\delta} (w_{n'_{r' + v'}} - c) 
  \mathds{1}_{w_{n'_{r' + v' + 1}} > c}
\dots \mathds{1}_{w_{n'_{r' + s'}} > c}  
]    \right|
\end{align*}
and 
\begin{align*}
& \left| \mathbb{E} [ h_R(x_{n_1}) \dots h_R(x_{n_r})
  g_{\delta} (w_{n_{r+1}} - c) 
  \dots g_{\delta} (w_{n_{r + v}} - c) 
  \mathds{1}_{w_{n_{r + v + 1}} > c}
\dots \mathds{1}_{w_{n_{r + s}} > c}   \right.
\\ & \quad \times  h_R(x_{n'_1})   \dots h_R(x_{n'_{r'}})
  g_{\delta} (w_{n'_{r'+1}} - c) 
  \dots g_{\delta} (w_{n'_{r' + v'}} - c) 
  \mathds{1}_{w_{n'_{r' + v' + 1}} > c}
\dots \mathds{1}_{w_{n'_{r' + s'}} > c}  
]  
\\ & -   \mathbb{E} [ h_R(x_{n_1}) \dots h_R(x_{n_r})
  g_{\delta} (w_{n_{r+1}} - c) 
  \dots g_{\delta} (w_{n_{r + v}} - c) 
  g_{\delta} (w_{n_{r + v+1}} - c) 
\dots \mathds{1}_{w_{n_{r + s}} > c}    
\\ & \left.\quad \times   h_R(x_{n'_1})   \dots h_R(x_{n'_{r'}})
  g_{\delta} (w_{n'_{r'+1}} - c) 
  \dots g_{\delta} (w_{n'_{r' + v'}} - c) 
  \mathds{1}_{w_{n'_{r' + v' + 1}} > c}
\dots \mathds{1}_{w_{n'_{r' + s'}} > c}  
]    \right|
\end{align*}
for integers $r \geq 0$, $s \geq 1$, $0 \leq v \leq s-1$,  
$n_1, \dots n_{r+s} \leq a$, $r' \geq 0$, $s' \geq 0$, $0 \leq v' \leq s'$,  
$n'_1, \dots n'_{r'+s'} \geq b$, such that $r + s$ and $r' + s'$ are between $1$ and $p$. 

We bound the differences of expectations or product of expectations 
by using again H\"older inequality, stationarity and the bounds $|h_R(x)| \leq |x|$, $|g_{\delta} (x)| \leq 1$. Moreover, we observe that by assumptions on $g_{\delta}$, 
$|g_{\delta} (x-c)- \mathds{1}_{x > c}|$ is always bounded by $1$, and vanishes as soon as $x-c \notin [-\delta,\delta]$:
$$|g_{\delta} (x-c)- \mathds{1}_{x > c}| \leq \mathds{1}_{|x-c| \leq \delta}.$$
We deduce that the expectations of $X$ and $Y$ are modified, for each replacement of a factor $\mathds{1}_{x_n > c}$ by $g_{\delta}(x_n- c)$, 
by at most a quantity of the form
\begin{align*} \mathbb{E} [ |x_{m_1}| | x_{m_2} | & 
\dots |x_{m_r}| \mathds{1}_{|w_m - c| \leq \delta}]
 \leq \prod_{j=1}^r \left(\mathbb{E} [ |x_{m_j}|^{r+1}] \right)^{1/(r+1)}
\left( \mathbb{P}[ |w_m - c| \leq \delta] \right)^{1/(r+1)}
\\ & \leq K_5(p) (1 + C(0)^{(p-1)/2})
(\min \{1, \delta (C(0)- C(1))^{-1/2}\})^{1/p},
\end{align*}
for some integers $m_1, \dots, m_r, m$, $r \leq p-1$, and $K_5(p) > 0$ depending only on $p$.

For the expectation 
of $XY$, we get a bound of its modification 
by a quantity of the form
$$\mathbb{E} [ |x_{n_1}| | x_{n_2} |  
\dots |x_{n_r}||x_{n'_1}| | x_{n'_2} |  
\dots |x_{n'_{r'}}|
\mathds{1}_{|w_n - c| \leq \delta}] $$
for $ r \leq p-1$, $ r' \leq p$, which is controlled by the same method. 
Overall, each replacement of $\mathds{1}_{w_n > c}$ by $g_{\delta} (w_n -c)$ gives a modification of 
$\mathbb{E} [XY] - \mathbb{E}[X] \mathbb{E}[ Y]$ by
at most 
$$K_6(p) (1 + C(0)^{p - 1/2} )
 (\delta (C(0)- C(1))^{-1/2})^{1/2p},$$
$K_6(p) > 0$ depending only on $p$. 

After doing all the $\mathcal{O}(p)$ changes of $x_n$ by $h_R(x_n)$ and of $\mathds{1}_{w_n > c}$ by $g_{\delta}(w_n -c)$  successively, 
we then get a total modification of 
$\mathbb{E} [XY] - \mathbb{E}[X] \mathbb{E}[ Y]$ by
at most 
\begin{align*}
& \mathcal{O}(p) \left( K_4(p) R^{-1} (1 + C(0)^{1/2 + p}) + K_6(p) (1 + C(0)^{p - 1/2} )
 (\delta (C(0)- C(1))^{-1/2})^{1/2p} \right)
 \\ \leq\,& 
K_7(p) (1 + C(0)^{p - 1/2} )
 (R^{-1}(1 + C(0)) + (\delta (C(0)- C(1))^{-1/2})^{1/2p}) ,
 \end{align*}
where $K_7(p) > 0$ depends only on $p$.

The second step of the proof consists in 
applying Proposition
\ref{proposition approximation bound}
in order to bound the covariance
of $X$ and $Y$ after modification 
of all factors $x_n$ by $h_R (x_n)$ and 
all factors $\mathds{1}_{w_n > c}$ by 
$g_{\delta} (w_n - c)$. In other words, 
we now bound $\mathbb{E}[X_{R, \delta} Y_{R, \delta}] - \mathbb{E}[X_{R, \delta}] \mathbb{E}[Y_{R, \delta}]$ for 
$$X_{R,\delta} = h_R(x_{n_1})  
\dots h_R(x_{n_r})   g_{\delta} (w_{n_{r+1}} -c)
\dots g_{\delta} (w_{n_{r+s}} -c)$$
and 
$$Y_{R,\delta} = h_R( x_{n'_1}) \dots 
\dots h_R(x_{n'_{r'}})   g_{\delta}(w_{n'_{r'+1}}- c) 
\dots  g_{\delta}(w_{n'_{r'+s'}}- c). $$

Keeping the notation of Proposition \ref{proposition approximation bound}, we take Gaussian random vectors
$$X_1 := (x_{n_1}, \dots, x_{n_r}, w_{n_{r+1}},\dots,  
w_{n_{r+s}}),$$
$$X_2 :=(x_{n'_1}, \dots, x_{n'_{r'}}, w_{n'_{r'+1}}, \dots, 
w_{n'_{r'+s'}}),$$
and functions
$$\Phi_1 (t_1, \dots, t_{r+s}) 
:= h_R(t_1) \dots h_R(t_r) g_{\delta}(t_{r+1}-c) \dots g_{\delta} (t_{r+s}-c),$$
$$\Phi_2 (t_1, \dots, t_{r'+s'}) 
:= h_R(t_1) \dots h_R(t_{r'}) g_{\delta}(t_{r'+1}-c) \dots g_{\delta} (t_{r'+s'}-c).$$
By assumption, $X_1$ and $X_2$ have dimension at most $p$, i.e. $d_1, d_2 \leq p$. 
Since $g_{\delta}$ takes values in $[0,1]$, $h_{R}$ takes values in 
$[-2R, 2R]$, $|g_{\delta}'| = \mathcal{O}(\delta^{-1})$ and $|h_R'| = \mathcal{O}(1)$, 
we deduce that each partial derivative of $\Phi_1$ and $\Phi_2$ is dominated by $(2R)^{p-1} \delta^{-1}$. 
We have that $\Phi_1 (0)$ and $\Phi_2 (0)$
are in $[0,1]$, since they vanish as soon as there is one factor $h_R$ involved and $g_{\delta}$ takes values in $[0,1]$. We get, for $s \in \{1,2\}$, 
$$N(\Phi_s) = |\Phi_s(0)| + \sup_{x \in \mathbb{R}^{d_s}} \| \nabla \Phi_s(x) \| = \mathcal{O} (
\sqrt{p} (2R)^{p-1} \delta^{-1}). 
$$
The coordinates 
of $X_1$ and $X_2$ 
are of the form $x_n$ or $w_n$, and then 
their variance is $C(0)$ or
$$\mathbb{E} [ w_n^2] 
= \mathbb{E} [ x_n^2 + x_{n-1}^2 
- 2 x_n x_{n-1}] = 2 (C(0) - C(1)) 
\leq 4 C(0)
$$
due to the assumption $C(0)> |C(1)|$. Hence, 
$$
M:= \max\{ \operatorname{Var}((X_1)_j), \operatorname{Var}((X_2)_k) : 1 \leq j \leq d_1,\, 1 \leq k \leq d_2 \} \leq 4 C(0).
$$ 
Moreover, the covariance of $x_n$ or 
$w_n$ for $n \leq a$ with 
$x_{n'}$ or $w_{n'}$ with $n \geq b$ is 
at most four times the supremum of 
$C(\ell)$ for $\ell \geq b-a-1$. We deduce that 
$$\mu := \max\{ |\operatorname{Cov}((X_1)_j, (X_2)_k)| : 1 \leq j \leq d_1,\, 1 \leq k \leq d_2 \} \leq 4 \sup_{\ell \geq b-a-1} 
|C(\ell)|,$$
since $X_1$ involves only variables of index at most $a$ and $X_2$ involves variables of index at least $b$. 

From Proposition \ref{proposition approximation bound}, we deduce 
\begin{align*}
&  |\mathbb{E} [ X_{R,\delta} Y_{R, \delta}]
- \mathbb{E}[X_{R, \delta}] \mathbb{E}[Y_{R,\delta}] |
\\  =\,& |\mathbb{E} [ \Phi_1 (X_1) \Phi_2(X_2)] - \mathbb{E} [ \Phi_1(X_1)] \mathbb{E} [ \Phi_2 (X_2)] |
\\ = \,& \mathcal{O}\left( p^3 
 (2R)^{2p - 2} \delta^{-2}  
\left(  \sup_{\ell \geq b-a-1} 
|C(\ell)| \right)^{1/2} (1 +   C(0))^{1/2} \right)\,.
\end{align*}
 For 
$$X = x_{n_1} x_{n_2} 
\dots x_{n_r}   \mathds{1}_{w_{n_{r+1} > c}}
\dots \mathds{1}_{w_{n_{r + s}} > c},$$
$$Y = x_{n'_1} x_{n'_2} 
\dots x_{n'_{r'}}   \mathds{1}_{w_{n'_{r'+1} > c}}
\dots \mathds{1}_{w_{n'_{r' + s'}} > c},$$
we deduce, from the bounds proven above, that 
\begin{align*}
& |\mathbb{E} [ XY]- \mathbb{E}[X] \mathbb{E}[Y]| 
\\ \leq \,& K_7(p) (1 + C(0)^{p - 1/2} )
 (R^{-1}(1 + C(0)) + (\delta (C(0)- C(1))^{-1/2})^{1/2p})
 \\ & + \mathcal{O}\left( p^3 
 (2R)^{2p - 2} \delta^{-2}  
\left(  \sup_{\ell \geq b-a-1} 
|C(\ell)| \right)^{1/2} (1 +   C(0))^{1/2} \right).
\end{align*}
Now, we simplify the bound by choosing $R$ and $\delta$. Since $|C(1)| < C(0)$ and hence $C(0) - C(1) \leq 2 C(0)$, 
\begin{align*} & (1 + C(0)^{p-1/2})(R^{-1}(1 + C(0)) +  
 (\delta(C(0)- C(1))^{-1/2})^{1/2p} )
 \\ =\,& (1 + C(0)^{p-1/2}) (C(0) - C(1))^{-1/4p}
( \delta^{1/2p} + R^{-1} ( 1 + C(0)) (C(0) - C(1))^{1/4p})
\\ \leq \,& (1 + C(0)^{p-1/2})
(C(0) - C(1))^{-1/4p}
( \delta^{1/2p} + R^{-1}(1 + C(0)) (2 C(0))^{1/4p})\,. 
\end{align*} 
Since 
$$(1 + C(0)^{p-1/2}) \leq  2(1 + C(0)^{p-1/2 + 1 + 1/4p})$$
and 
$$(1 + C(0)^{p-1/2}) (1 + C(0)) (2 C(0))^{1/4p}$$
is also dominated by 
$1 + C(0)^{p-1/2 + 1 + 1/4p}$, because 
$C(0)^b \leq C(0)^a + C(0)^c$ for $0 \leq a \leq b \leq c$, the quantity above is dominated, for a given value of $p$, by  
\begin{align*} 
&(R^{-1}+\delta^{1/2p})(1 + C(0)^{p-1/2 + 1 + 1/4p})
(C(0) - C(1))^{-1/4p}\\
= \,& (R^{-1}+\delta^{1/2p}) (1+ C(0)^{p+1/2 + 1/4p})
(C(0) - C(1))^{-1/4p}\,.
\end{align*}
For $(1 + C(0))^{1/2}$ in the $\mathcal{O}$ term, we bound it by
\begin{align*}
(1 + C(0))^{1/2}
= \, &(1 + C(0))^{1/2} (2 C(0))^{1/4p} (2 C(0))^{-1/4p}\\
\leq\,& (1 + C(0))^{1/2} (2 C(0))^{1/4p} (C(0) - C(1))^{-1/4p}\,,
\end{align*}
which is dominated by
$$2(1 + C(0)^{1/2 + 1/4p})
(C(0) - C(1))^{-1/4p} \leq 2(1 + C(0)^{p+1/2 + 1/4p})
(C(0) - C(1))^{-1/4p}.
$$
We deduce 
\begin{align*}
& |\mathbb{E} [ XY]- \mathbb{E}[X] \mathbb{E}[Y]| 
\\ \leq \,& 
K_8(p) (C(0)-C(1))^{-1/4p} (1 + C(0)^{p+1/2 + 1/4p})
 \\ & \times \left( R^{-1} 
+ \delta^{1/2p} 
+   R^{2p-2} \delta^{-2}  \left(\sup_{\ell \geq b-a-1} |C(\ell)| \right)^{1/2} \right)
\end{align*}
where $K_8(p) > 0$ depends only on $p$.
We now choose 
$$R = 1+ \left(\sup_{\ell \geq b-a-1} |C(\ell)| \right)^{-1/(12p-2)}$$
and $\delta = R^{-2p}$, 
which gives 
\begin{align*}
&\quad R^{-1} 
+ \delta^{1/2p} 
+   R^{2p-2} \delta^{-2}  \left(\sup_{\ell \geq b-a-1} |C(\ell)| \right)^{1/2} 
\\ &  = 2 R^{-1}
+  R^{2p-2} R^{4p}(R - 1)^{-(6p-1)}
\\ & \leq 2 R^{-1}
+  R^{6p-2} (R - 1)^{-(6p-1)}
\\ & \leq 2 (R-1)^{-1}
+   (1 + (R-1))^{6p-2} (R-1)^{-(6p-1)}
\\ & \leq 2 (R-1)^{-1} 
+ 2^{6p-2} (1 + (R-1)^{6p-2}) 
 (R-1)^{-(6p-1)}
 \\ & \leq (2 + 2^{6p-2}) (R-1)^{-1}
 + 2^{6p-2} (R-1)^{-(6p-1)}
 \\ & = (2 + 2^{6p-2})
\left[\left(\sup_{\ell \geq b-a-1} |C(\ell)| \right)^{1/(12p-2)}
+ \left(\sup_{\ell \geq b-a-1} |C(\ell)| \right)^{1/2} \right]
\\ & \leq (2 + 2^{6p-2})
\left(\sup_{\ell \geq b-a-1} |C(\ell)| \right)^{1/(12p-2)} \left( 1 + C(0)^{1/2- 1/(12p-2)} \right). 
\end{align*}
We then get 
\begin{align*}
    & |\mathbb{E} [ XY] - \mathbb{E} [X] \mathbb{E}[Y] | 
\\ \leq \,& K_9(p) (C(0)- C(1))^{-1/4p} (1 + C(0)^{p+1 + 1/4p - 1/(12p-2)} )\left(\sup_{\ell \geq b-a-1} |C(\ell)| \right)^{1/(12p-2)}  
\end{align*}
for $K_9(p) > 0$ depending only on $p$.

\end{proof}
We deduce from the result above 
a bound on the joint cumulants of 
terms appearing on 
the PRSA analysis. 
\begin{prop}[Bounds for cumulants of PRSA terms] \label{boundcumulants}
Let \( (x_n)_{n \in \mathbb{Z}} \) be a centered stationary Gaussian process with covariance function \( C(k) = \mathbb{E}[x_0 x_k] \) such that $C(0) > |C(1)|$, and define increments \( w_n := x_n - x_{n-1} \). 
For integers $p, r, s \geq 1$, $a < b$, let $Y_1, \dots, Y_r$ by random variables equal to the product of at most $p$ random factors among $x_n$ and $\mathds{1}_{w_n > c}$ for $n \leq a$, i.e. they have the same form as $X$ in Proposition \ref{covarianceboundPRSA}, and let $Y_{r+1}, \dots, Y_{r+s}$ be random variables equal to the product of at most $p$ factors among $x_n$ and $\mathds{1}_{w_n > c}$ for $n \geq b$, i.e. they have the same form as $Y$ in Proposition \ref{covarianceboundPRSA}. 
Then, the absolute value of the joint cumulant of order $r+s$ of
the random variables $Y_1, \dots, Y_{r+s}$ is well-defined and bounded by
$$K (p, r,s, C(0), C(1)) \left(\sup_{\ell \geq b-a-1} |C(\ell)| \right)^{1/(12p(r+s)-2)}$$
for $K (p, r,s, C(0), C(1)) > 0$ depending only on $p$, $r$, $s$, $C(0)$ and $C(1)$. 
\end{prop}

\begin{proof}
Since Gaussian variables are in $L^q$ 
for all $q \geq 1$, the variables 
$Y_1, \dots, Y_{r+s}$ have joint moments 
of all orders, and then joint cumulants of all orders. 
The joint cumulant of $Y_1, \dots, Y_{r+s}$ can be written as 

$$\sum_{\Pi \in \mathcal{P} (\{1, \dots, r+s\}) } \alpha(\Pi) \prod_{A \in \Pi} 
\mathbb{E} \left[ \prod_{j \in A} 
Y_{j}\right],
$$
$ \mathcal{P} (\{1, \dots, r+s\})$ being the set of all partitions $\Pi$ of 
$\{1, \dots, r+s\}$, $A$ being a subset 
of $\{1, \dots, r+s\}$ in the partition $\Pi$, and $\alpha (\Pi)$ being a coefficient depending only on $\Pi$ (and then on $r+s$). 
Now, 
\begin{equation}\sum_{\Pi \in \mathcal{P} (\{1, \dots, r+s\}) } \alpha(\Pi) \prod_{A \in \Pi} 
\mathbb{E} \left[ \prod_{j \in A, 1 \leq j \leq r} 
Y_{j}\right] \, \mathbb{E} \left[ \prod_{j \in A, r+1 \leq j \leq r+s} 
Y_{j}\right] = 0, \label{vanishcumulant}
\end{equation}
because the left-hand side of this equality is equal to the joint cumulant 
of $Y'_1, \dots, Y'_{r+s}$, where 
$(Y'_1, \dots, Y'_r)$ and 
$(Y'_{r + 1}, \dots, Y'_{r+s})$
are two independent families of random variables, respectively distributed as 
$(Y_1, \dots, Y_r)$ and 
$(Y_{r + 1}, \dots, Y_{r+s})$. Indeed, 
the joint cumulant of the union of two independent families of random variables vanishes if it is well-defined. 
The joint cumulant of $Y_1, \dots, Y_{r+s}$ is then equal to the sum of the modifications of the left-hand side 
of \eqref{vanishcumulant} when 
we replace one by one each factor 
$$\mathbb{E} \left[ \prod_{j \in A, 1 \leq j \leq r} 
Y_{j}\right] \, \mathbb{E} \left[ \prod_{j \in A, r+1 \leq j \leq r+s} 
Y_{j}\right] $$
by 
$$\mathbb{E} \left[ \prod_{j \in A} 
Y_{j}\right]$$
for each $\Pi \in \mathcal{P} (\{1, \dots, r+s\})$ and $A \in \Pi$. 
Using triangle and H\"older inequalities, 
we deduce that the joint cumulant of 
$Y_1, \dots, Y_{r+s}$ has a modulus bounded 
by 
\begin{align*} 
& \sum_{\Pi \in \mathcal{P} (\{1, \dots, r+s\}) } |\alpha(\Pi)|\sum_{A \in \Pi} 
\prod_{B \in \Pi, B \neq A}
\prod_{j \in B} 
\left(\mathbb{E} [ |Y_j|^{\operatorname{Card}(B)}]\right)^{1/\operatorname{Card}(B)}
\\ & 
\times\left| \mathbb{E} \left[ \prod_{j \in A} 
Y_{j}\right]   - \, \mathbb{E} \left[ \prod_{j \in A, 1 \leq j \leq r} 
Y_{j}\right] \, \mathbb{E} \left[ \prod_{j \in A, r+1 \leq j \leq r+s} 
Y_{j}\right] \right| \,.
\end{align*}
Since $Y_j$ is bounded by a product 
of at most $p$ Gaussian variables of 
variance dominated by $C(0)$,
and since the number of 
terms of the sums and factors in the products, as well as $|\alpha(\Pi)|$ and the cardinality of $B$, are bounded by a quantity depending only on $r$ and $s$, 
we deduce that the joint cumulant 
 of 
$Y_1, \dots, Y_{r+s}$ has a modulus bounded 
by $K_1 (p, r,s, C(0))$ times the supremum, for the subsets $A$ of $\{1, \dots, r+s\}$, of 
$$\left| \mathbb{E} \left[ \prod_{j \in A} 
Y_{j}\right]   - \, \mathbb{E} \left[ \prod_{j \in A, 1 \leq j \leq r} 
Y_{j}\right] \, \mathbb{E} \left[ \prod_{j \in A, r+1 \leq j \leq r+s} 
Y_{j}\right] \right| $$
where $K_1 (p, r,s, C(0)) > 0$ 
depends only on $p, r, s, C(0)$. 
The last quantity is bounded by using Proposition \ref{covarianceboundPRSA},
with $p$ replaced by $p(r+s)$ since
the product of $Y_j$ for $j \in A$ has at most $p(r+s)$ factors of the form $\mathds{1}_{w_n > c}$ or $x_n$. 
This provides the estimate of 
Proposition \ref{boundcumulants}.

\end{proof}

\subsection{Law of large numbers for stationary Gaussian processes}

The entries of the PRSA-averaged signal \( z_{n,L} \) are computed by dividing a weighted sum of signal values (numerator) by the number of hinge points (denominator). While the denominator appears to be a normalization factor, both the numerator and denominator involve the same hinge-point selection, leading to statistical dependence.
For stationary Gaussian process with covariance going to zero, we can prove a law of large numbers. 
\begin{prop} \label{lawlargenumbers}
Let $(x_n)_{n \in \mathbb{Z}}$ be a centered stationary Gaussian process with covariance function $C(k) = \mathbb{E} [ x_0 x_k]$, and define 
increments $w_n := x_n - x_{n-1}$. 
We assume that the covariance function $C$ tends to zero at infinity. Then, 
for all $\ell \in \mathbb{Z}$,
$$\frac{1}{2n+1} \sum_{m=-n}^n 
x_{m+\ell} \mathds{1}_{w_m > c} 
\underset{n \rightarrow \infty}{\longrightarrow} \mathbb{E}[ x_{\ell} \mathds{1}_{w_0 > c} ]$$
and 
$$\frac{1}{2n+1} \sum_{m=-n}^n 
 \mathds{1}_{w_m > c} 
\underset{n \rightarrow \infty}{\longrightarrow}\mathbb{E}[ \mathds{1}_{w_0 > c} ] = \mathbb{P} (w_0 > c)$$
in probability.
If $C(0) > |C(1)|$, for any $\ell\in \mathbb{N}$, we deduce the convergence in probability of
$$ \frac{  \sum_{m=-n}^n 
x_{m+\ell} \mathds{1}_{w_m > c} }{ \sum_{m=-n}^n 
 \mathds{1}_{w_m > c} } $$ to 
$$\zeta_{\ell} := \frac{\mathbb{E}[ x_{\ell} \mathds{1}_{w_0 > c} ]}{\mathbb{E}[ \mathds{1}_{w_0 > c} ] }\,.
$$
Then, for any $L \in \mathbb{N}$, the convergence in probability of $z_{n,L}$ to $(\zeta_{\ell})_{- L \leq \ell \leq L}$. 
\end{prop}
\begin{proof}
Let $s$ be equal to $0$ or $1$ and fix $\ell\in \{-L,\ldots,L\}$.  
We have 
$$\mathbb{E} \left[ \frac{1}{2n+1} \sum_{m=-n}^n 
x^s_{m+\ell} \mathds{1}_{w_m > c} 
\right] = \mathbb{E} [x_{\ell}^s \mathds{1}_{w_0 > c}].$$
By Chebyshev's inequality, it is then sufficient to check that 
$$\operatorname{Var} 
\left(\frac{1}{2n+1} \sum_{m=-n}^n 
x^s_{m+\ell} \mathds{1}_{w_m > c} \right) \underset{n \rightarrow \infty}{\longrightarrow} 0, $$
i.e. using linearity of the covariance and stationarity of $(x_n)_{n \in \mathbb{Z}}$, 
$$\frac{1}{(2n+1)^2} 
\sum_{m = -2n}^{2n} (2n+1 - m) \operatorname{Cov} 
 (x_{\ell}^s \mathds{1}_{w_0 > c}, 
 x_{m+\ell}^s \mathds{1}_{w_m > c})
 \underset{n \rightarrow \infty}{\longrightarrow} 0
 $$
 since there are $2n+1 - m$ couples of integers $(m_1, m_2)$ in $\{-n, \dots, n\}$ such that $m_2 - m_1 = m$. 
Applying Proposition \ref{covarianceboundPRSA} with $p = 2$, 
$a = |\ell|$, $b = m - |\ell|$
if $m \geq 2 |\ell| + 1$, and $p=2$, $a = m + |\ell|$, $b = -|\ell|$
if $m \leq -2 |\ell| - 1$, 
we deduce that for $|m| \geq 2 |\ell| +1 $, 
$$|\operatorname{Cov} 
 (x_{\ell}^s \mathds{1}_{w_0 > c}, 
 x_{m+\ell}^s \mathds{1}_{w_m > c})| 
 \leq K(2) (C(0)- C(1))^{-1/8}(1 + C(0)^{3+1/8-1/22}) 
  \left( \underset{ k \geq |m| - 2 |\ell| - 1}{\sup} |C(k)| \right)^{1/22}.$$
For $|m| \leq 2 |\ell|$, we use 
Cauchy-Schwarz inequality and stationarity to get
\begin{align*}|\operatorname{Cov} 
 (x_{\ell}^s \mathds{1}_{w_0 > c}, 
 x_{m+\ell}^s \mathds{1}_{w_m > c})| 
& \leq 
\sqrt{\operatorname{Var} 
 (x_{\ell}^s \mathds{1}_{w_0 > c} ) 
 \operatorname{Var} 
 (x_{m+\ell}^s \mathds{1}_{w_m > c} ) 
 }
\\ &  = \operatorname{Var} 
 (x_{\ell}^s \mathds{1}_{w_0 > c} ) 
 \leq \mathbb{E} [x_{\ell}^{2s} \mathds{1}_{w_0 > c} ]
 \\ & \leq \mathbb{E} [ x_{\ell}^{2s}]
 \leq \mathbb{E} [ 1 +  x_{\ell}^{2}]
 = 1 + C(0)\,.  
\end{align*}
Since $L$ is fixed, it is then sufficient to prove 
$$\frac{1}{(2n+1)^2} 
\sum_{m = -2n}^{2n} (2n+1 - m) 
 \left( \underset{ k \geq |m| - 2 |\ell| - 1}{\sup} |C(k)| \right)^{1/22}\underset{n \rightarrow \infty}{\longrightarrow} 0\,,
 $$
which can be shown, 
via the change of variable 
$m = \lfloor (2n+1) t \rfloor$, by controlling the integral 
$$\int_{-1}^{1} 
\left( 1 - \frac{\lfloor (2n+1) t \rfloor}{2n+1} \right)
\left( \underset{ k \geq |\lfloor (2n+1) t \rfloor| - 2 |\ell| - 1}{\sup} |C(k)| \right)^{1/22} dt.
$$
The quantity to integrate is 
bounded by $C(0)^{1/22}$, and tend to zero 
at each $t$ different from zero, by 
the assumption that $C(k) \rightarrow 0$
when $k \rightarrow \infty$. The desired convergence is then deduced from dominated convergence.

\end{proof}

The value of $\zeta_{\ell}$ can be computed explicitly. 
Keeping the notation of Proposition \ref{lawlargenumbers}, we get the following:

\begin{prop}
Assume that $C(0) > 0$ and $C(0) > C(1)$. Then for any $\ell \in \mathbb{Z}$,
\begin{align}\label{expression of zeta in LNN}
\zeta_{\ell} = 
\frac{ C(\ell) - C(\ell+1) }{ \sqrt{4\pi (C(0) - C(1))} } e^{-c^2 / 4(C(0) - C(1))} 
\cdot Q\left( \frac{c}{\sqrt{2(C(0) - C(1))}} \right)^{-1},
\end{align}
where $Q(x) := \int_x^\infty \frac{1}{\sqrt{2\pi}} e^{-u^2/2} \, du$ is the standard Gaussian tail function.
\end{prop}

Note that $\zeta_\ell$ follows the pattern of $C(\ell)-C(\ell+1)$, the difference of covariance structure, and it depends on the threshold $c$ via a linear scaling $e^{-c^2 / 4(C(0) - C(1))} \cdot Q\left( \frac{c}{\sqrt{2(C(0) - C(1))}} \right)^{-1}$. In other words, in practice this formula can be applied to recover the covariance structure of the stationary random process via cumsum of $z_{n,L}$ with various $c$.
\begin{proof}
By assumption,
\begin{align}\label{control var of wm}
\operatorname{Var} (w_m) 
= \mathbb{E} [ w_m^2] 
= \mathbb{E} [ x_m^2] + \mathbb{E}[ x_{m-1}^2]
- 2 \mathbb{E} [ x_m x_{m-1} ]
= 2 (C(0) - C(1)) 
\end{align}
is strictly positive.
We get 
$$\mathbb{P} (w_0 >c) =  Q\left(\frac{c}{\sqrt{ 2 (C(0) - C(1))}}\right).
$$
Moreover, 
\begin{align*}
\mathbb{E}\left[ x_{\ell} \mathds{1}_{w_0 > c}\right]&\,=\mathbb{P} ( w_0 > c) \mathbb{E} [ x_{\ell} | w_0 >  c ]\\
&= Q\left(\frac{c}{\sqrt{ 2(C(0) - C(1))}}\right) \mathbb{E} [  x_{\ell} | w_0 >  c ].
\end{align*}
Since $(x_{\ell}, w_0)$ is a centered Gaussian vector
with the same covariance matrix as 
$$\left( \frac{\operatorname{Cov} ( w_0, x_{\ell}) }{ \operatorname{Var} (w_0)} w_0  + y, w_0 \right),
$$
where $y$ is a centered Gaussian variable, independent of $w_0$, with variance 
$$\operatorname{Var}(x_{\ell}) - 
\frac{\operatorname{Cov}^2 ( w_0, x_{\ell})}{ \operatorname{Var} (w_0)}, 
$$
we deduce that the two vectors have the same joint distribution. 
Denote $N$ to be a standard Gaussian random variable. We have 
\begin{align*}
&\mathbb{E} [ x_{\ell} | w_0 > c]\\
= \,& \mathbb{E} \left[ \frac{\operatorname{Cov} ( w_0, x_{\ell}) }{ \operatorname{Var} (w_0)} w_0  + y
\; \big| w_0 > c\right]
=  \frac{\operatorname{Cov} ( w_0, x_{\ell}) }{ \operatorname{Var} (w_0)}
\mathbb{E} [ w_0 | w_0 > c]  
\\ =  \,& \frac{\operatorname{Cov} ( w_0, x_{\ell}) }{ \sqrt{\operatorname{Var} (w_0)} }
\mathbb{E} [ N | N > c / \sqrt{\operatorname{Var} (w_0)} ]  
\\ =  \,& \frac{\operatorname{Cov} ( x_0 - x_{-1}, x_{\ell}) }{
\sqrt{2 (C(0) - C(1))}} 
\mathbb{E} [ N | N > c / \sqrt{2 (C(0) - C(1))} ]  
\\ = \,&  \frac{ C(\ell) - C(\ell+1) }{
\sqrt{2 (C(0) - C(1))}} 
\mathbb{E} \left[ N  \mathds{1}_{N > c / \sqrt{2 (C(0) - C(1))}} \right]  Q(c/ \sqrt{ 2(C(0) - C(1)) })^{-1}\,.
\end{align*}
We then get 
$$\mathbb{E} \left[ x_{\ell} \mathds{1}_{w_0 > c}
\right] 
= \frac{ C(\ell) - C(\ell+1) }{
\sqrt{2 (C(0) - C(1))}} 
\mathbb{E} \left[ N  \mathds{1}_{N > c / \sqrt{2 (C(0) - C(1))}} \right]. $$
The last expectation is 
$$\frac{1}{\sqrt{2 \pi}} 
\int_{c / \sqrt{2 (C(0) - C(1))}}^{\infty} x e^{-x^2/2} dx 
= \frac{1}{\sqrt{2 \pi}} e^{- c^2/ 4 ( C(0) - C(1))} \,,
$$
which implies 
$$\mathbb{E} \left[  x_{\ell} \mathds{1}_{w_0 > c}
\right] = \frac{ C(\ell) - C(\ell+1) }{
\sqrt{4 \pi (C(0) - C(1))}} e^{- c^2/ 4 ( C(0) - C(1))}. $$
We thus conclude that
$$\frac{\mathbb{E} \left[  x_{\ell} \mathds{1}_{w_0 > c}
\right]}{\mathbb{E} \left[  \mathds{1}_{w_m > c} \right] }
= \frac{ C(\ell) - C(\ell+1) }{
\sqrt{4 \pi (C(0) - C(1))}} e^{- c^2/ 4 ( C(0) - C(1))}
Q\left(\frac{c}{\sqrt{ 2(C(0) - C(1)) }}\right)^{-1}.
$$
\end{proof}

\begin{rem}
The limit $\zeta_{\ell}$ has a natural interpretation: it is the conditional expectation $\mathbb{E}[x_{m+\ell} \mid w_m > c]$ for any $m \in \mathbb{Z}$, while the denominator is the probability $\mathbb{P}(w_m > c)$. Since $w_m = x_m - x_{m-1}$ is Gaussian with variance $2(C(0) - C(1))$, the tail probability $\mathbb{P}(w_m > c)$ is given by $Q\left( c / \sqrt{2(C(0) - C(1))} \right)$.

The conditional mean $\mathbb{E}[x_{m+\ell} \mid w_m > c]$ can be computed explicitly using the linear regression formula for jointly Gaussian variables. The result depends linearly on $C(\ell)$ and $C(\ell+1)$, reflecting the lag-$\ell$ and lag-$(\ell+1)$ covariances between $x_{m+\ell}$ and $x_m, x_{m-1}$, which generate $w_m$.

This shows that the PRSA average converges to a weighted average of future signal values, conditioned on a past increment exceeding the threshold $c$. The dependence on $c$ and the covariance function determines how strongly hinge point selection filters the signal.
\end{rem}

\subsection{Cumulant Estimates and CLT Preparations}

The computations presented above justify the asymptotic expansion of the PRSA statistic $z_{n,L}$
around its limit and establish the probabilistic structure necessary to prove a Central Limit Theorem. By decomposing the ratio and controlling error terms through a Taylor expansion, we reduce the problem to establishing a joint CLT for the numerator and the denominator. To this end, we compute second moments and bound higher-order cumulants using decay properties of the covariance function $C(k)$. In particular, we show that if $C(k)$  decays sufficiently fast, then the cumulants vanish in the limit, which ensures asymptotic Gaussianity. These results provide the technical foundation for the formal CLT in the next subsection. More precisely, the following holds: 
\begin{prop}  \label{convergencecumulanttozero} Let $(x_n)_{n \in \mathbb{Z}}$ be a centered stationary Gaussian process with covariance function $C(k) = \mathbb{E}[x_0 x_k]$ and define increments $w_n := x_n - x_{n-1}$. We assume that $C(0) > |C(1)|$ and that $C$ decays faster than any power at infinity, i.e. for all $A \geq 0$, $C(k) k^A$ tends to zero when $k \rightarrow \infty$. Then, for each real-valued
sequence $(\alpha_{\ell})_{\ell \in \mathbb{Z}}$, 
 finitely many of the $\alpha_{\ell}$'s being different from zero, for each $\beta \in \mathbb{R}$, and for each integer $p \geq 3$, the
 $p$-th cumulant of 
 $$\frac{1}{\sqrt{2n+1}} 
 \left(\sum_{m = -n}^n  \left( \beta +  \sum_{\ell \in \mathbb{Z}}
 \alpha_{\ell} x_{m + \ell} \right) \mathds{1}_{w_m > c} 
\right)  $$
tends to zero when $n$ goes to infinity. 
\end{prop}
\begin{proof}
The $p$-th cumulant is equal to 
$(2n+1)^{-p/2}$ times the sum, for $m_1, \dots, m_p$
between $-n$ and $n$, of the joint cumulant 
of the $p$ variables 
$$\left(\left( \beta +  \sum_{\ell \in \mathbb{Z}}
 \alpha_{\ell} x_{m_j + \ell} \right) \mathds{1}_{w_{m_j} > c}  \right)_{1 \leq j \leq p}.$$
Since $\alpha_{\ell} = 0$ for all but finitely many values of $\ell$, multilinearity of the joint cumulants shows that is it enough to prove that 
for each $\ell_1, \dots, \ell_p \in \mathbb{Z}$, 
$s_1, \dots, s_{p} \in \{0,1\}$, 
$$(2n+1)^{-p/2}\sum_{-n \leq m_1, \dots, m_p \leq n}
\kappa_p \left( \left( x_{m_j + \ell_j}^{s_j} \mathds{1}_{w_{m_j} > c}\right)_{1 \leq j \leq p} \right)
\underset{n \rightarrow \infty}{\longrightarrow} 0,$$
where $\kappa_p$ denotes the joint cumulant of order $p$.  
Let $\mu$ be the difference between the smallest and the largest of the indices $m_1, \dots, m_p$, 
and $\lambda$ the maximum of $|\ell_j|$ for $1 \leq j \leq p$. By the pigeonhole principle, 
if we rank $m_1, \dots, m_p$ in nondecreasing order, one of the gaps is at least $\mu/(p-1)$.
Hence, there exist $a$ and $b$ such that 
some of the variables $x_{m_j + \ell_j}^{s_j} \mathds{1}_{w_{m_j} > c}$ involve only indices $m_j$ and $m_j + \ell_j$ smaller than or equal to $a$, and all the other variables involve indices 
larger than or equal to $b$, for 
$$b-a \geq \frac{\mu}{p-1} - 2 \lambda.$$
By Proposition \ref{boundcumulants}, we deduce 
\begin{align*}  & \left|\kappa_p \left( \left( x_{m_j + \ell_j}^{s_j} \mathds{1}_{w_{m_j} > c}\right)_{1 \leq j \leq p} \right) \right| 
\\  \leq \,& \sup_{1 \leq r \leq p-1} K (2, r,p-r, C(0), C(1)) \left(\sup_{k \geq \mu/(p-1) - 2 \lambda - 1} |C(k)| \right)^{1/(24p-2)}\,.
\end{align*}
We sum this bound on all $(m_1, \dots, m_p)$
in $([-n,n] \cap \mathbb{Z})^p$. 
For each integer $m \geq 0$, choosing a $p$-tuple 
with $\mu = m$ gives at most $2n+1$ possibilities for $m_1$, and for a given choice of $m_1$, at most $2m+1$ possibilities for each of the integers $m_2, \dots, m_p$. 
We deduce 
\begin{align*}
& \left|(2n+1)^{-p/2}\sum_{-n \leq m_1, \dots, m_p \leq n}
\kappa_p \left( \left( x_{m_j + \ell_j}^{s_j} \mathds{1}_{w_{m_j} > c}\right)_{1 \leq j \leq p} \right) \right|
\\\leq \,& (2n+1)^{-p/2} \sum_{m \geq 0} 
(2n+1) (2m+1)^{p-1} 
\\ & \times \sup_{1 \leq r \leq p-1} K (2, r,p-r, C(0)) \left(\sup_{k \geq m/(p-1) - 2 \lambda - 1} |C(k)| \right)^{1/(24p-2)}.
\end{align*}
The rapid decay of $C(k)$ when $k \rightarrow \infty$ ensures that the sum in $m$ converges, 
and then the left-hand side is dominated by 
$(2n+1)^{1 - p/2}$, which tends to zero for $p \geq 3$. 
\end{proof}
 \subsection{The central limit theorem}
From the result above on cumulants, we deduce the CLT for PRSA statistics
\begin{thm}[Central Limit Theorem for PRSA statistics]\label{central limit theorem for prsa}
Let $(x_n)_{n \in \mathbb{Z}}$ be a centered stationary Gaussian process with covariance function \( C(k) = \mathbb{E}[x_0 x_{k}] \), and define \( w_n := x_n - x_{n-1} \). We assume that $C(0) > 0$, $C(0) > |C(1)|$ and $C(k) k^A \rightarrow 0$ when $k \rightarrow \infty$, for all $A\geq 0$. Fix \( L \in \mathbb{N} \) and \( c \in \mathbb{R} \). Then, the PRSA statistic
$$
z_{n,L} = \left( \frac{  \sum_{m=-n}^n x_{m+\ell} \mathds{1}_{\{w_m > c\}} }
{   \sum_{m=-n}^n \mathds{1}_{\{w_m > c\}} } \right)_{-L \leq \ell \leq L}
$$
satisfies the central limit theorem:
$$
\sqrt{2n+1} \left(z_{n,L} - (\zeta_\ell)_{-L \leq \ell \leq L} \right) \xrightarrow{d} \mathcal{N}(0, V_L),
$$
where \( \zeta_\ell \) is the limiting value given in Proposition ~\ref{lawlargenumbers}, and
$ \mathcal{N} (0, V_L)$ is a Gaussian vector taking values in $\mathbb{R}^{\{-L,-L+1, \dots,L\}}$, with mean $0$ and covariance matrix
$$V_L = (\operatorname{Cov}_{\ell, \ell'})_{-L \leq \ell, \ell' \leq L},$$
for 
$$
\operatorname{Cov}_{\ell, \ell'} =
\mathbb{P}(w_0 > c)^{-2} \sum_{h \in \mathbb{Z}} 
\operatorname{Cov} ((x_{\ell} - \zeta_{\ell}) \mathds{1}_{w_0 > c}, 
(x_{\ell'+h} - \zeta_{\ell'}) \mathds{1}_{w_h > c} ),
$$
where the infinite sum is absolutely convergent, and $\mathbb{P}(w_0 > c) > 0$. 
\end{thm}
\begin{proof}
We first prove CLT for 
the random vector $((X_{n,\ell})_{-L \leq \ell \leq L}, Y_n)\in \mathbb{R}^{2L+2}$, for 
$$X_{n,\ell} := \frac{1}{\sqrt{2 n+1}}  \sum_{m=-n}^n \left( x_{m+\ell} \mathds{1}_{w_m > c} - \mathbb{E} [x_{\ell} \mathds{1}_{w_0 > c}] \right) $$
and
$$Y_n :=  \frac{1}{\sqrt{2 n+1}}  \sum_{m=-n}^n  (\mathds{1}_{w_m > c} - \mathbb{P} (w_0 > c) ).$$
This vector is centered by construction.  
Moreover, by Proposition \ref{convergencecumulanttozero}, all linear combinations of components of the vector $((X_{n,\ell})_{-L \leq \ell \leq L}, Y_n)$ have cumulants of order larger than or equal to $3$ which tend to zero when $n$ goes to infinity.
In order to get a CLT, it is 
then sufficient to check the convergence of the covariance matrix of the random vector, since Gaussian distributions are characterized by their moments. 
For $-L \leq \ell, \ell' \leq L$, the covariance of $X_{n, \ell}$ and $X_{n, \ell'}$ is equal to 
$$\frac{1}{2n+1}\sum_{-n \leq m_1, m_2 \leq n} \operatorname{Cov} (x_{m_1 + \ell} \mathds{1}_{w_{m_1} > c},x_{m_2 + \ell'} \mathds{1}_{w_{m_2} > c} ),$$
and then, by letting $h = m_2 - m_1$ and using stationarity, we have
$$
\text{Cov}(X_{n,\ell},X_{n,\ell'})=\sum_{-2n \leq h \leq 2n} \frac{2n+1-h}{2n+1} \operatorname{Cov} (x_{  \ell} \mathds{1}_{w_0 > c},x_{h+ \ell'} \mathds{1}_{w_h > c} ).$$
From the rapid decay of $C$ at infinity and Proposition \ref{covarianceboundPRSA}, we deduce 
that 
$$\sum_{h \in \mathbb{Z}} |\operatorname{Cov} (x_{  \ell} \mathds{1}_{w_0 > c},x_{h+ \ell'} \mathds{1}_{w_h > c} )| < \infty,$$
 and then, by dominated convergence, 
$$\operatorname{Cov} (X_{n, \ell}, X_{n, \ell'} ) 
\underset{n \rightarrow \infty}{\longrightarrow} 
\sum_{h \in \mathbb{Z}} \operatorname{Cov} (x_{  \ell} \mathds{1}_{w_0 > c},x_{h+ \ell'} \mathds{1}_{w_h > c} ) =: \operatorname{Cov}^{(XX)}_{\ell, \ell'}\,.
$$
Similarly, 
$$\operatorname{Cov} (X_{n, \ell}, Y_n ) 
\underset{n \rightarrow \infty}{\longrightarrow} 
\sum_{h \in \mathbb{Z}} \operatorname{Cov} (x_{  \ell} \mathds{1}_{w_0 > c},\mathds{1}_{w_h > c} ) 
=: \operatorname{Cov}^{(XY)}_{\ell}
$$
and 
$$\operatorname{Var} ( Y_n ) 
\underset{n \rightarrow \infty}{\longrightarrow} 
\sum_{h \in \mathbb{Z}} \operatorname{Cov} ( \mathds{1}_{w_0 > c},\mathds{1}_{w_h > c} ) =:\operatorname{Cov}^{(YY)}\,.
$$
We then have proven that 
$((X_{n,\ell})_{-L \leq \ell \leq L}, Y_n)$
converges to a centered Gaussian vector with covariance matrix given by  $\operatorname{Cov}^{(XX)}_{\ell, \ell'}$, $\operatorname{Cov}^{(XY)}_{\ell}$ and $\operatorname{Cov}^{(YY)}$.  
Now, we have 
$$z_{n, L} = \left(\frac{ (2n+1) \mathbb{E}[x_{\ell} \mathds{1}_{w_0 > c} ] + \sqrt{2n+1} X_{n, \ell}}{
(2n+1) \mathbb{P} (w_0 > c) + \sqrt{2n+1} Y_n} \right)_{-L \leq \ell \leq L}, $$
$$z_{n, L} - (\zeta_{\ell})_{-L \leq \ell \leq L}= \left(\frac{   \mathbb{E}[x_{\ell} \mathds{1}_{w_0 > c} ] +  X_{n, \ell}/\sqrt{2n+1}}{
  \mathbb{P} (w_0 > c) + Y_n/\sqrt{2n+1}} - 
  \frac{ \mathbb{E}[x_{\ell} \mathds{1}_{w_0 > c} ] }{
  \mathbb{P} (w_0 > c)}\right)_{-L \leq \ell \leq L}, $$

and then 
$$\sqrt{2n+1} (z_{n, L} - (\zeta_{\ell})_{- L \leq \ell \leq L} ) 
=  \left( \frac{X_{n, \ell} \mathbb{P}(w_0 > c) - 
Y_n \mathbb{E} [x_{\ell} \mathds{1}{w_0 > c}]} { \mathbb{P} (w_0 > c) (\mathbb{P} (w_0 > c) + Y_n/\sqrt{2n+1})} \right)_{-L \leq \ell  \leq L}. 
$$
From the convergence of $((X_{n,\ell})_{-L \leq \ell \leq L}, Y_n)$, we deduce that the numerator converges to a centered Gaussian vector with 
covariance matrix given by 
\begin{align*} & \left( \mathbb{P} (w_0 > c)^2 
\operatorname{Cov}^{(XX)}_{\ell, \ell'} 
-\mathbb{P} (w_0 > c) \mathbb{E} [ x_{\ell} \mathds{1}_{w_0 > c}] \operatorname{Cov}^{(XY)}_{\ell'}
\right. \\ & \left. - \mathbb{P} (w_0 > c) \mathbb{E} [ x_{\ell'} \mathds{1}_{w_0 > c}] \operatorname{Cov}^{(XY)}_{\ell}
+ \mathbb{E} [ x_{\ell} \mathds{1}_{w_0 > c}]
\mathbb{E} [ x_{\ell'} \mathds{1}_{w_0 > c}]\operatorname{Cov}^{(YY)} \right)_{-L \leq \ell, \ell' \leq L}, 
\end{align*}
which can be rewritten as 
$$ \left( \mathbb{P}(w_0 > c)^2 \sum_{h \in \mathbb{Z}} 
\operatorname{Cov} ((x_{\ell} - \zeta_{\ell}) \mathds{1}_{w_0 > c}, 
(x_{\ell'+h} - \zeta_{\ell'}) \mathds{1}_{w_h > c} )
\right)_{- L \leq \ell, \ell' \leq L}
$$
since 
$$\zeta_{\ell} = \frac{\mathbb{E} [x_{\ell} \mathds{1}{w_0 > c}]}{\mathbb{P} (w_0 > c)}.$$
Since $Y_n$ converges in distribution, 
the denominator tends to $\mathbb{P}(w_0 > c)^2$ in probability. Since $C(0) > |C(1)|$, $\mathbb{P}(w_0 > c) > 0$, and by Slutsky's theorem, the ratio converges in distribution to a centered Gaussian vector of covariance matrix 
$$ \left( \mathbb{P}(w_0 > c)^{-2} \sum_{h \in \mathbb{Z}} 
\operatorname{Cov} ((x_{\ell} - \zeta_{\ell}) \mathds{1}_{w_0 > c}, 
(x_{\ell'+h} - \zeta_{\ell'}) \mathds{1}_{w_h > c} )
\right)_{- L \leq \ell, \ell' \leq L}
$$

\end{proof}

\begin{rem}
While a nonstationary $x_i$ is not the focus of this paper, we provide some preliminary results when $x_i$ is a simple random walk; that is, when $x_0 = 0$ and $(w_p)_{p \in \mathbb{Z}}$ are i.i.d. Gaussian (to start with).
Since the hinge points are ``regularly distributed'' (see the law of large numbers), following the notation in Lemma \ref{lemma 4.1}, we have 
$$ 
\frac{H(p -\ell -1, n)}{H(-n-1, n)} \simeq \frac{n- p + \ell + 1}{2n+1} \simeq \frac{n-p}{2 n}
$$ 
for fixed $\ell$ and $n$ going to infinity. In this case, we can expect 
$$
z_{n}(\ell) \simeq x_0 + \sum_{p= 1}^{n + \ell} w_p \frac{n-p}{2n} 
-  \sum_{p= -n + \ell + 1}^{0} w_p \frac{n+p}{2n}.$$
We would get 
$z_{n}(\ell)$ not much dependent on $\ell$ at first order, and 
\begin{align*}
z_{n}(\ell) &\, \simeq \mathcal{N} \left(0,\, \sum_{p=1}^n \left(\frac{n-p} {2n}  \right)^2+ \sum_{p=-n}^0 \left(\frac{n+p} {2n}  \right)^2  \right)\\
%\mathcal{N} \left(0, \sum_{q=0}^{n-1} \left(\frac{q} {2n}  \right)^2+ \sum_{q=0}^n \left(\frac{q} {2n}  \right)^2  \right)
&\,=\mathcal{N} \left(0,  \frac{n (n-1)(2n-1)/6}{ 4n^2} 
+ \frac{n (n+1)(2n+1)/6}{ 4n^2}  \right)\,,
\end{align*}
which has a leading order  
$$z_{n}(\ell) \simeq \mathcal{N} (0, n/6).$$
The vector $z_{n, L}$, at leading order, is expected to be approximately constant, and equal to a normal variable of mean $0$ and variance $n/6$. 

The situation is different if we compare different values of $\ell$. At the limit, in the coordinates of $z_{n, L}$, we can show that there is a jump between the coordinates of negative indices and the coordinates of nonnegative indices, of order $1$; whereas the other variations of coordinates have order $1/\sqrt{n}$. The coordinates themselves are of order $\sqrt{n}$. Indeed, when we compare two consecutive values $\ell$ and $\ell + 1$, we get 
\begin{align*} 
&z_{n}( \ell + 1)- z_{n}(\ell) \\
  = \,&
\sum_{p = 1}^{n+ \ell} w_p \frac{H(p -\ell -2 , n) -H(p -\ell -1 , n) }{H(-n-1, n)}
+ w_{n+ \ell + 1} \frac{H(n-1, n)}{H(-n-1, n)}
\\ &  -  \sum_{p= -n + \ell + 2}^{0} w_p \frac{ H( -n-1, p- \ell - 2) -H( -n-1, p- \ell - 1) }{ H(-n-1, n)} + w_{-n+ \ell + 1} \frac{   H( -n-1, -n) }{ H(-n-1, n)} 
\\ =  \,&\frac{1}{ H(-n-1, n)} \left( \sum_{p = 1}^{n+ \ell} w_p \mathds{1}_{w_{p -\ell -1 } > c } +  \sum_{p= -n + \ell + 2}^{0} w_p \mathds{1}_{w_{p -\ell -1 } > c } +
w_{n+ \ell + 1}  \mathds{1}_{w_{n} > c}
+  w_{-n+ \ell + 1}  \mathds{1}_{w_{-n} > c} \right) 
\\ = \,& \frac{1}{ H(-n-1, n)} \sum_{p= -n + \ell +1 }^{n + \ell+1} w_p \mathds{1}_{w_{p- \ell - 1} > c} =  \frac{1}{ H(-n-1, n)} \sum_{p= -n }^{n } w_{p + \ell + 1} \mathds{1}_{w_p > c}.  
\end{align*} 
For $\ell = -1$, one gets 
$$
\frac{1}{ H(-n-1, n)} \sum_{p= -n }^{n } w_{p } \mathds{1}_{w_p > c}\,,
$$
which, by the law of large numbers, converges almost surely to 
$$\frac{\mathbb{E} [ w_0 \mathds{1}_{w_0 > c} ]}{ \mathbb{P} (w_0 > c)} 
= \mathbb{E} [ w_0 | w_0 > c].$$
For $\ell \neq -1$, there is a central limit theorem, with a variance of order $1/n$. Indeed, 
the $L^2$ norm of the sum above (which in fact corresponds to sums of increments of a martingale) is 
$$\sum_{-n \leq p, q \leq n} \mathbb{E} [w_{p + \ell + 1} w_{q + \ell + 1}
\mathds{1}_{w_p > c}\mathds{1}_{w_q > c}].$$
All terms are uniformly bounded. For $p \neq q, q - \ell - 1$, the index 
$p+ \ell + 1$ is different from $q + \ell + 1$ (because $p \neq q$), different from $p$ (because $\ell \neq -1$) and different from $q$ (because $p \neq q - \ell - 1$). Hence, 
$w_{p+\ell + 1}$ is independent of the other factors, centered, and then the expectation of the product of all factors is zero. In the double sum in $p$ and $q$, there is at most two non-zero terms for each value of $q$, and then $\mathcal{O}(n)$ non-zero terms in total. The double sum is then $\mathcal{O}(n)$. 
We deduce that
$$\mathbb{E} [(z_{n}( \ell + 1)- z_{n}(\ell))^2] = \mathcal{O}(1/n). $$
The increments of $z_{n}( \ell)$ tend in probability to zero, except the one from $z_{n}( -1)$ to $z_{n}( 0)$, which converges to $\mathbb{E} [ w_0 | w_0 > c]$.

\end{rem}

\subsection{Numerical simulation}\label{section numerics}

See Figure \ref{fig:PRSAstochastic} for an example of various stationary random processes, including $x^{(1)}$ following a Gaussian white noise, $x^{(2)}$ following ARMA(2,1) with autoregressive coefficient (AR) $(0.01, 0.15)$ and moving average (MA) coefficient $-0.15$, $x^{(3)}$ following  ARMA(2,1) with AR coefficient $(0.1, -0.85)$ and MA coefficient $0.5$, and $x^{(4)}$ following ARMA(4,2) with AR coefficient $(0.01, 0.01, 0,-0.9)$ and MA coefficient $(0.2,-0.5)$. We realize $n=8,000,000$ points and choose $L=100$. We observe that the established law of large numbers accurately captures the empirical behavior of $z_{n,L}$. Moreover, although all datasets are generated from simple stationary ARMA models, which intuitively are not quasi-periodic by definition, the resulting outputs exhibit diverse patterns and dynamics. In particular, some realizations display intricate, seemingly quasi-periodic structures with amplitude modulation at fixed frequencies, as exemplified by the ARMA(4,2) case.

Mathematically, by interpreting the covariance function as the inverse Fourier transform of the power spectrum of a stationary random process, the PRSA output can be viewed as the detected quasi-periodic structure of the process if the covariance function is composed of incommensurate frequencies. However, even if it exists, this notion of quasi-periodicity is inherently statistical in nature and differs fundamentally from that defined for deterministic functions. This result further underscores the necessity of caution when interpreting the outputs of PRSA.

\begin{figure}[!hbt]
\centering
\includegraphics[trim=0 0 0 0, clip, width=\textwidth]{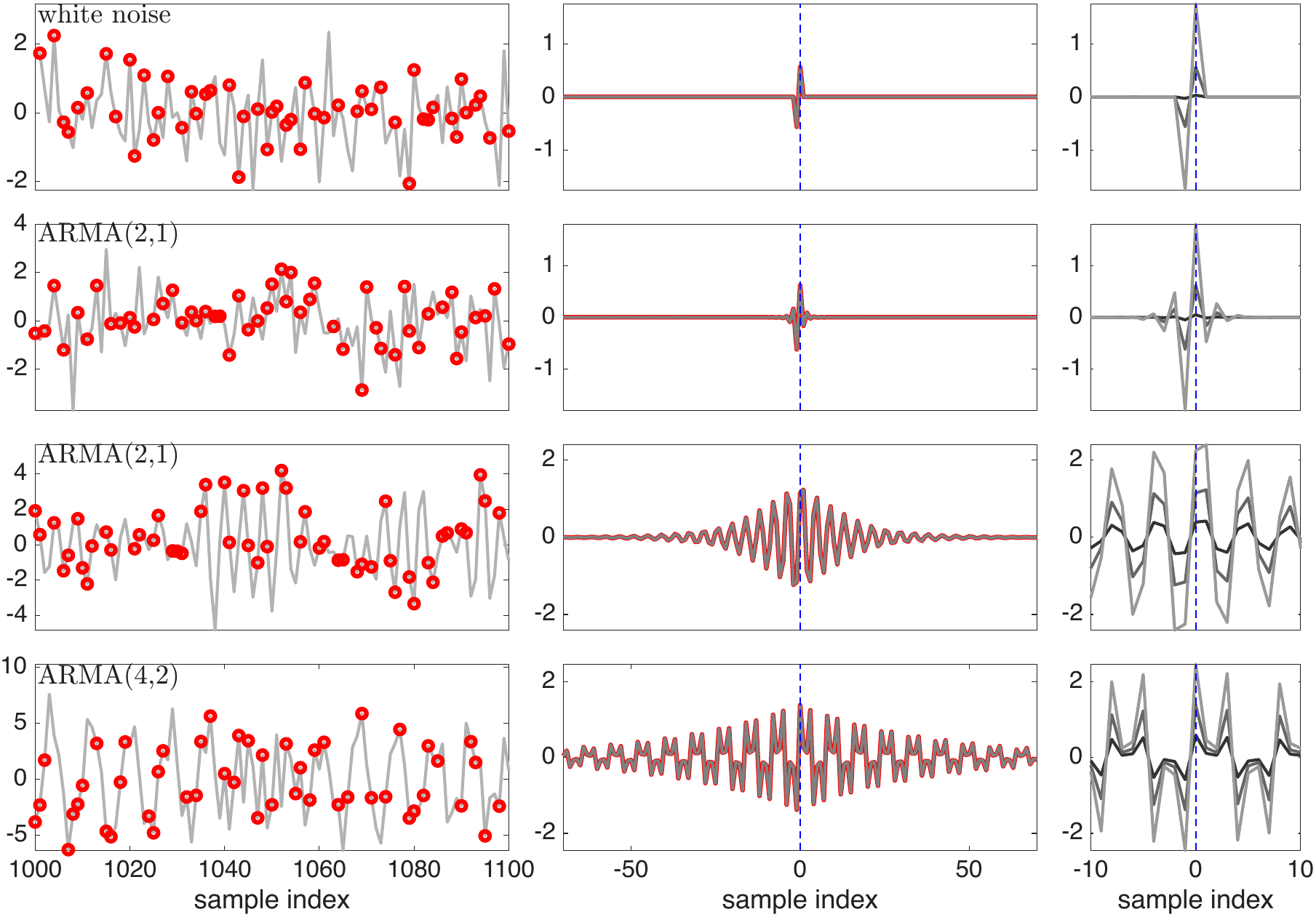} 
\caption{\label{fig:PRSAstochastic} PRSA of stationary stochastic signals. Left column: a realization of a random process indicated in the upper left corner is shown in gray, with the points $w_n>0$ marked in red; that is, $c=0$.  Middle column: the predicted law of large number of PRSA with $c=0$ is the red curve, and the PRSA output
$z_{n,L}$ with $c=0$ is superimposed as the dark gray curve. The vertical dashed blue line indicated the middle point of $z_{n,L}$. Right column: the PRSA outputs
$z_{n,L}$ with $c=-3,0,3$ are superimposed as the black, dark gray and light gray curves, respectively. The vertical dashed blue line indicated the middle point of $z_{n,L}$. As predicted by the theorem, these curves differ only by a constant scaling factor.}
\end{figure}

{ 
\section{Conclusion and future direction}
While PRSA has been widely applied to the analysis of biomedical signals, our investigations using both a simple two-harmonic deterministic model and a stationary random process model highlight the need for caution when interpreting PRSA outputs. The seemingly quasiperiodic patterns in the output may not faithfully represent the underlying dynamics unless the generating mechanism is clearly specified, an essential consideration if the ultimate scientific goal is to understand the process that produced the observed time series. Further study is therefore needed to elucidate what information is truly embedded in the signal, in line with the scientific rationale motivating the use of PRSA.

From a mathematical standpoint, several interesting directions remain open for future work. First, one may consider the nonnull setting, in which the observed time series takes the form $f+X$, where $f$ satisfies the two-harmonic model and $X$ is a stationary random process. In practical applications, PRSA is expected to recover information about $f$ while minimizing the impact $X$; that is, the PRSA output of $f+X$ should be close to that of $f$. We need to quantify how close is close, and its statistical behavior.

Second, while this work adopts the classical mathematical definition of quasiperiodicity \cite{levitan1982almost,arnol2013mathematical,amerio2013almost}, in practice the term “quasiperiodic” is often used more broadly to describe deterministic signals with slowly varying amplitude and frequency, or random processes exhibiting recurrent structures such as cyclostationarity \cite{gardner1975characterization}. Motivated by this practical perspective, a promising extension is to analyze PRSA adapting the framework of the adaptive harmonic model, where $f(t)=\cos(2\pi t) + A(t)\cos(2\pi\phi(t))$, with $A(t)>0$ smooth and slowly varying, and $\phi(t)$ smooth, monotonically increasing and satisfying $\phi'(t)\in (0,1)$ with slow variation, or its generalization to non-sinusoidal oscillation \cite{HTWu2013}. The formulation of $A(t)$ and $\phi'(t)$ captures amplitude and frequency modulations commonly observed in real-world signals. Note that if $A(t)=A>0$ and $\phi'(t)=\xi\in (0,1)$, then it is reduced to the two harmonic model considered in this paper. Another possibility is adapting the almost periodic function framework \cite{EoM_almost_periodic_function}, which captures the periodicity different from the adaptive harmonic model, which is phenomenological.

Third, the behavior of PRSA under non-Gaussian or nonstationary noise warrants further analysis. Ideally, a central limit theorem should still hold under appropriate regularity conditions. Since the techniques in the present study rely heavily on Gaussianity and stationarity, new mathematical tools will be required to handle this more general case. If a suitable Gaussian approximation can be established, one could further develop a bootstrap-based inference framework for practical applications.

Finally, an important direction is the theoretical analysis of multivariable PRSA \cite{schumann2008bivariate,bauer2009bprsa,muller2012bivariate}. Extending PRSA to high-dimensional time series raises questions analogous to those in the univariate setup, with extra information to explore regarding the interaction among channels. Collectively, these directions aim to establish a mathematically rigorous foundation for understanding and applying PRSA to address fundamental scientific questions.
}

\bibliographystyle{plain}
\bibliography{ref}

\end{document}